\documentclass[runningheads,a4paper,envcountsame,envcountsect]{llncs}

\usepackage{amssymb}
\usepackage{graphicx}

\usepackage{url}

\usepackage[utf8]{inputenc}

\usepackage[english]{babel}
\usepackage{amsfonts}
\usepackage{amssymb}
\usepackage{amsmath}
\usepackage{latexsym}
\usepackage{color}
\usepackage{subfigure}
\usepackage{enumerate}

\usepackage{hyperref}  

\usepackage{ifpdf}

\graphicspath{{figures/}}

\DeclareMathOperator    \conv           {conv}

\DeclareMathOperator    \intr                   {int}

\DeclareMathOperator    \proj           {proj}

\DeclareMathOperator    \relint         {rel\,int}

\DeclareMathOperator    \verts          {vert}

\DeclareMathOperator    \Aff {Aff}  

\newcommand{\old}[1]{{}}

\newcommand{\bb}{\mathbb}

\newcommand{\R}{\bb R}
\newcommand{\Q}{\bb Q}
\newcommand{\Z}{\bb Z}

\newcommand{\C}{\bb C}

\newcommand\st{\mid}

\def\ve#1{\mathchoice{\mbox{\boldmath$\displaystyle\bf#1$}}
{\mbox{\boldmath$\textstyle\bf#1$}}
{\mbox{\boldmath$\scriptstyle\bf#1$}}
{\mbox{\boldmath$\scriptscriptstyle\bf#1$}}}

\let\saveendproof=\endproof
\def\endproof{\qed\saveendproof}

\newcommand{\I}{\mathcal{P}}  
\newcommand{\diag}{\smallsetminus}
\newcommand{\Idiag}{\I_{q,\diag}}
\newcommand{\Ivert}{\I_{q,\,\mid\,}}
\newcommand{\Ihor}{\I_{q,-}}

\renewcommand{\P}{\mathcal{P}}
\newcommand{\E}{\mathcal{E}}
\newcommand{\G}{\mathcal{G}}
\renewcommand{\S}{\mathcal{S}}

\newcommand{\x}{{\ve x}}
\newcommand{\y}{{\ve y}}
\renewcommand{\v}{{\ve v}}
\newcommand{\g}{{\ve g}}
\renewcommand{\u}{{\ve u}}
\renewcommand{\a}{{\ve a}}
\newcommand{\f}{{\ve f}}
\newcommand{\0}{{\ve 0}}
\newcommand{\m}{{\ve m}}
\newcommand{\p}{{\ve p}}

\newcommand{\B}{B}

\def\st{\mid}
\newenvironment{psmallmatrix}{\left(\smallmatrix}{\endsmallmatrix\right)}
\newcommand\ColVec[2]{\begin{psmallmatrix}#1\\#2\end{psmallmatrix}}

\title{Equivariant Perturbation in \\Gomory and Johnson's Infinite Group
  Problem.\\ II. The Unimodular Two-Dimensional Case}
\titlerunning{Equivariant Perturbation II}

\author{Amitabh Basu \and
Robert Hildebrand \and
Matthias K\"oppe}

\institute{Dept.\ of Mathematics, University of California, Davis\\
  \texttt{abasu@math.ucdavis.edu}, \texttt{rhildebrand@math.ucdavis.edu}, \texttt{mkoeppe@math.ucdavis.edu}}

\date{\today}

\usepackage[normalem]{ulem}

\begin{document}

\maketitle
\begin{abstract}
  We give an algorithm for testing the extremality of a large class of
  minimal valid functions for the two-dimensional infinite group problem. 
\end{abstract}
\section{Introduction}


\subsection{The group problem}

Gomory's \emph{group problem}~\cite{gom} is a central object in the study of strong
cutting planes for integer linear optimization problems.  One considers an
abelian (not necessarily finite) group $G$, written additively, and studies
the set of functions $s \colon G \to \R$ satisfying the following constraints:  
                \begin{gather}
                        \sum_{\ve r \in G} \ve r\, s(\ve r) \in \ve f + S \tag{IR} \label{GP} \\
                        s(\ve r) \in \mathbb{Z}_+ \ \ \textrm{for all $\ve r \in G$}  \notag\\
                        s \textrm{ has finite support}, \notag
                \end{gather}
where $\ve f$ is a given element in $G$, and $S$ is a subgroup of $G$; so $\ve
f + S$ is the coset containing the element $\ve f$.  
We will be concerned with the so-called {\em infinite group
  problem} \cite{infinite,infinite2}, where $G =\R^k$ is taken to be the group
of real $k$-vectors under addition, and $S= \Z^k$ is the subgroup of the
integer vectors. 
We are interested in studying the convex hull $R_{\ve f}(G,S)$ of all
functions satisfying the constraints in~\eqref{GP}. Observe that $R_{\ve f}(G,S)$ is
a convex subset of the infinite-dimensional vector space
$\mathcal{V}$ of functions $s \colon G \to \R$ with finite support. 


Any linear inequality in
$\mathcal{V}$ is given by a pair $(\pi, \alpha)$ where $\pi$ is a function
$\pi\colon G \to \R$ (not necessarily of finite support) and $\alpha \in
\R$. The linear inequality is then given by $\sum_{\ve r \in G} \pi(\ve
r)s(\ve r) \geq
\alpha$; the left-hand side is a finite sum because $s$ has finite
support. Such an inequality is called a {\em valid inequality} for $R_{\ve f}(G,S)$
if $\sum_{\ve r \in G} \pi(\ve r)s(\ve r) \geq \alpha$ for all $s \in R_{\ve f}(G,S)$.  It is
customary to concentrate on those valid inequalities for which $\pi \geq 0$;
then we can choose, after a scaling, $\alpha = 1$. Thus, we only focus on
valid inequalities of the form $\sum_{\ve r \in G} \pi(\ve r)s(\ve r) \geq 1$ with $\pi
\geq 0$. Such functions $\pi$ will be termed {\em valid functions} for
$R_{\ve f}(G,S)$. 

A valid function $\pi$ for $R_{\ve f}(G,S)$ is said to be
\emph{minimal} for $R_{\ve f}(G,S)$ if there is no valid function $\pi' \neq \pi$
such that $\pi'(\ve r) \le \pi(\ve r)$ for all $\ve r \in G$. For every valid
function $\pi$ for $R_{\ve f}(G,S)$, there exists a minimal valid function $\pi'$
such that $\pi' \leq \pi$ (cf.~\cite{bhkm}), and thus non-minimal valid
functions are redundant in the description of $R_{\ve f}(G,S)$.  Minimal functions for $R_{\ve f}(G,S)$ were characterized by Gomory for
finite groups $G$ in~\cite{gom}, and later for $R_{\ve f}(\R,\Z)$ by Gomory and
Johnson~\cite{infinite}. We state these results in a unified notation in the
following theorem. 

A function $\pi\colon G \rightarrow \mathbb{R}$ is \emph{subadditive} if
$\pi(\x + \y) \le \pi(\x) + \pi(\y)$ for all $\x,\y \in G$. We say that  $\pi$ is
\emph{symmetric} if $\pi(\x) + \pi(\f - \x) = 1$ for all $\x \in G$.

\begin{theorem}[Gomory and Johnson \cite{infinite}] \label{thm:minimal} Let
  $\pi \colon G \rightarrow \mathbb{R}$ be a non-negative function. Then $\pi$
  is a minimal valid function for $R_{\ve f}(G,S)$ if and only if $\pi(\ve r) = 0$ for
  all $\ve r\in S$, $\pi$ is subadditive, and $\pi$ satisfies the symmetry
  condition. (The first two conditions imply that $\pi$ is constant over any
  coset of $S$.)
\end{theorem}
        

\subsection{Characterization of extreme valid functions}

A stronger notion is
that of an {\em extreme function}. A~valid function~$\pi$ is \emph{extreme}
for $R_{\ve f}(G,S)$ if it cannot be written as a convex combination of two other
valid functions for $R_{\ve f}(G,S)$, i.e., $\pi = \frac{1}{2}\pi_1 +
\frac{1}{2}\pi_2$ implies $\pi = \pi_1 = \pi_2$.  Extreme functions are
minimal.
A tight characterization of extreme functions for $R_{\ve f}(\R^k,\Z^k)$ has eluded
researchers for the past four decades now, however, various specific
sufficient conditions for guaranteeing
extremality~\cite{bhkm,3slope,dey1,dey2,deyRichard,tspace} have been proposed.
The standard technique for showing extremality is as follows.  Suppose that $\pi =
\frac12\pi^1+\frac12\pi^2$, where $\pi^1,\pi^2$ are other (minimal) valid
functions.  All subadditivity relations that are tight for~$\pi$ are also
tight for $\pi^1,\pi^2$.  Then one uses a lemma of \emph{real
  analysis}, the so-called Interval Lemma introduced by Gomory and Johnson
in~\cite{tspace} or one of its variants.  
The Interval Lemma allows us to deduce certain affine
linearity properties that $\pi^1$ and $\pi^2$ share 
with~$\pi$.  This is followed by a finite-dimensional \emph{linear algebra} argument 
to establish uniqueness of~$\pi$, implying $\pi=\pi^1=\pi^2$, and thus the
extremality of~$\pi$.

Surprisingly, the \emph{arithmetic} (number-theoretic) aspect of the problem
has been largely overlooked, even though it is at the core of the
theory of the closely related \emph{finite} group problem.  
In \cite{basu-hildebrand-koeppe:equivariant}, the authors showed that this aspect
is the key for completing the classification of extreme functions.  
The authors studied the case $k=1$ and gave a
complete and algorithmic answer for the case of piecewise linear functions
with rational breakpoints in the set~$\frac1q\Z$. To capture the relevant
arithmetics of the problem, the authors studied sets of
additivity relations of the form $\pi(\ve t_i) + \pi(\y) = \pi(\ve t_i + \y)$ and 
$\pi(\x) + \pi(\ve r_i-\x) = \pi(\ve r_i)$, where the points $\ve t_i$ and $\ve r_i$ are
certain breakpoints of the function~$\pi$.  They give rise to the \emph{reflection group}~$\Gamma$ generated by the reflections
$\rho_{\ve r_i}\colon \x \mapsto \ve r_i-\x$ and translations~$\tau_{\ve t_i}\colon \y\mapsto
\ve t_i+\y$.  The natural action of the reflection group~$\Gamma$ on the set
of intervals delimited by the elements of~$\frac1q\Z$ transfers the affine
linearity established by  
the Interval Lemma on some interval~$I$ to a connected component of the orbit
$\Gamma(I)$. 
When this establishes affine linearity
of $\pi^1, \pi^2$ on all intervals where $\pi$ is affinely linear, one
proceeds with finite-dimensional linear algebra to decide extremality
of~$\pi$.  Otherwise, there is a way to perturb $\pi$ slightly to
construct distinct minimal valid functions $\pi^1 = \pi + \bar\pi$ and $\pi^2
= \pi - \bar\pi$, using any sufficiently small,
$\Gamma$-equivariant perturbation function (see \autoref{s:equivariant}), modified by restriction to a
certain connected component. 
\medbreak

\subsection{Contributions of the paper}

In the present paper, we continue the program
of~\cite{basu-hildebrand-koeppe:equivariant}.  We study a remarkable class of
minimal functions~$\pi$ of the two-dimensional infinite group problem
($k=2$).  Let $q$ be a positive integer.  Consider the
arrangement~$\mathcal H_q$ of all hyperplanes (lines) of the form
$ (0,1)\cdot \x = b$, $(1,0)\cdot \x = b$, and $(1,1)\cdot\x  = b$,
where $b \in \tfrac{1}{q}\Z$.  The complement of the arrangement~$\mathcal
H_q$ consists of two-dimensional cells, whose closures are the triangles
$T_0 = \frac1q \conv(\{ \begin{psmallmatrix}0\\0\end{psmallmatrix}
, \begin{psmallmatrix}1\\0\end{psmallmatrix}
, \begin{psmallmatrix}0\\1\end{psmallmatrix}
 \})$ and $T_1 = \frac1q \conv(\{\begin{psmallmatrix}1\\0\end{psmallmatrix}
, \begin{psmallmatrix}0\\1\end{psmallmatrix}
,
\begin{psmallmatrix}1\\1\end{psmallmatrix}
\})$ and their translates by elements of the lattice $\smash[t]{\frac1q\Z^2}$. 
We denote by $\P_q$ the collection of these triangles and the vertices and
edges that arise as intersections of the triangles.  Thus $\P_q$ is a
polyhedral complex that is a triangulation of the space~$\R^2$.
Within the polyhedral complex $\P_q$, let $\I_{q,0}$ be the set of 0-faces (vertices),
$\I_{q,1}$ be the set of 1-faces (edges), and $\I_{q,2}$ be the set of 2-faces
(triangles)
.  The sets of diagonal, vertical, and horizontal edges will be denoted by
$\Idiag$, $\Ivert$, and $\Ihor$, respectively.
We will use $\oplus$ and $\ominus$ to denote vector addition and subtraction
modulo $1$, respectively.  We use the same notation for pointwise sums and
differences of sets. By quotienting out by~$\Z^2$, we obtain a finite complex that
triangulates~$\R^2/\Z^2$; we still denote it by~$\P_q$.  

We call a function $\pi\colon \R^2\to \R$ \emph{continuous piecewise linear over
  $\P_q$} if it is an affine linear function on each of the triangles of~$\P_q$. 
We introduce the following notation.  For every $I \in \P_q$, 
the restriction $\pi|_I$ is an affine function, that is $\pi|_I(\x) = \m_I \cdot \x + b_I$ for
some $\m_I\in \R^2$, $b_I \in \R$.   We abbreviate $\pi|_I$ as
$\pi_I$.  

For a valid function $\pi$, we consider the set $E(\pi) = \{\, (\x,\y) \st \pi(\x) + \pi(\y) =
\pi(\x \oplus \y)\,\}$ of pairs $(\ve x, \ve y)$, for which the subadditivity
relations are tight.  
Because~$\mathcal P_q$ enjoys a strong unimodularity property (\autoref{lemma:I+J}), 
we can give a finite combinatorial representation of the set $E(\pi)$ using the
faces of~$\P_q$; this extends a technique
in~\cite{basu-hildebrand-koeppe:equivariant}.  
%
For faces $I, J, K \in \P_q$, let $$F(I,J,K) = \{\,(\x,\y) \in \R^2\times\R^2
\st \x \in I,\, \y \in J,\, \x \oplus \y \in K\,\}.$$ 
A triple $(I,J,K)$ of faces is called a \emph{valid
  triple} (\autoref{def:valid-triple}) if none of the sets $I$,
$J$, $K$ can individually decreased without changing the resulting set $F(I,J,K)$. 
Let $E(\pi, \P_q)$ denote the set of valid triples $(I,J,K)$ such that 
$$\pi(\x) + \pi(\y) = \pi(\x \oplus \y) \quad\text{for all}\quad (\x,\y) \in
F(I,J,K).$$ 
$E(\pi, \P_q)$ is partially ordered by letting $(I,J,K) \leq (I',J',K')$ if
and only if $I\subseteq I'$, $J\subseteq J'$, and $K\subseteq K'$. 
Let $E_{\max{}}(\pi, \P_q) 
$ be the set of all maximal
valid  triples of the poset $E(\pi, \P_q)$.  
Then $E(\pi)$ is exactly covered by the sets $F(I,J,K)$ for the maximal valid
triples $(I,J,K) \in E_{\max{}}(\pi, \P_q)$ (\autoref{lemma:covered-by-maximal-valid-triples}).


In the present paper, we will restrict ourselves to a setting without maximal valid
triples that include horizontal or vertical edges.
\begin{definition}  A continuous piecewise linear function $\pi$ on $\P_q$ is
  called \emph{diagonally constrained} if whenever $(I, J, K) \in E_{\max{}}(\pi, \P_q)$, then $I,J,K \in \I_{q,0} \cup \Idiag \cup \I_{q,2}$.
  %
\end{definition}

\begin{remark}
  Given a piecewise linear continuous valid function~$\zeta\colon \R\to\R$ for
  the one-dimensional infinite group problem,
  Dey--Richard \cite[Construction 6.1]{deyRichard} consider the
  function~$\kappa\colon 
  \R^2\to\R$, $\kappa(\ve x) =\zeta(\ve1\cdot\ve x)$, where $\ve1=(1,1)$, and show that $\kappa$ is
  minimal and extreme if and only if $\zeta$ is minimal and extreme,
  respectively. 
  If $\zeta$ has rational breakpoints in $\frac1q\Z$, then $\kappa$ belongs to
  our class of diagonally constrained continuous piecewise linear functions over $\P_q$.
\end{remark}

We prove the following main theorem.
\begin{theorem}\label{thm:main}
Consider the following problem.  
\begin{quote}
  Given a minimal 
  valid function $\pi$ for $R_f(\R^2,\Z^2)$ that is 
  piecewise linear continuous on~$\P_q$ and diagonally constrained, 
  decide if $\pi$ is extreme.
\end{quote}
There exists an algorithm for this problem that takes a number of elementary operations over the reals that is
bounded by a polynomial in $q$.
\end{theorem}

As a direct corollary of the proof of the theorem, we obtain the following result relating
the finite and infinite group problems.

\begin{theorem}\label{thm:1/4q}
Let $\pi$ be a minimal continuous piecewise linear function over $\P_q$ that is diagonally constrained.  
Then $\pi$ is extreme for $R_{\ve f}(\R^2, \Z^2)$ if and only if the restriction $\pi\big|_{\tfrac{1}{4q}\Z^2}$ is extreme for $R_{\ve f}(\frac{1}{4q} \Z^2, \Z^2)$.  
\end{theorem}

We conjecture that the hypothesis on $\pi$ being diagonally constrained can be
removed. 

\section{Real analysis lemmas}
\label{s:real-analysis}

For any element $\x \in\R^k$, $k \geq 1$, $\lvert \x \rvert$ will denote the
standard Euclidean norm. The proof of the following theorem appears in 
appendix~\ref{s:continuity}.

\begin{theorem}
\label{Theorem:functionContinuous}
If $\pi \colon \R^k \to \R$ is a minimal valid function, and $\pi = \frac{1}{2}\pi^1 + \frac{1}{2}\pi^2$, where $\pi^1, \pi^2$ are valid functions, then $\pi^1, \pi^2$ are both minimal. Moreover, if $\limsup_{\ve h\to 0} \frac{\lvert\pi(\ve h)\rvert}{\lvert \ve h\rvert} < \infty$, then this condition also holds for $\pi^1$ and $\pi^2$. This implies that $\pi, \pi^1$ and $\pi^2$ are all Lipschitz continuous.
\end{theorem}

The following lemmas are corollaries of a general version of the interval
lemma or similar real analysis arguments.  Proofs appear in appendix~\ref{s:generalized-interval-lemma}.

\begin{lemma}\label{cor:triangle+triangle}
Suppose $\pi$ is a continuous function and let $(I, J, K) \in E(\pi, \P_q)$ be a valid triple of triangles, i.e., $I,J,K \in \I_{q,2}$. Then $\pi$ is affine in $I, J, K$ with the same gradient.
\end{lemma}

\begin{lemma}\label{cor:triangle+diagonal}
Suppose $\pi$ is a continuous function and let $(I, J, K) \in E(\pi, \P_q)$
where $I\in \Idiag$, $J,K \in \Idiag \cup \I_{q,2}$. 
Then $\pi$ is affine in the diagonal direction in $I, J, K$,
i.e., there exists $c \in \R$  such that such that $\pi(\v + \lambda \ColVec{-1}{1}) = \pi(\v) + c\cdot\lambda$ for all $\v \in I$ (resp., $\v \in J$, $\v \in
K$) and $\lambda \in \R$ such that $\v + \lambda \ColVec{-1}{1} \in I$ (resp., $\v
+ \lambda \ColVec{-1}{1} \in J$, $\v + \lambda \ColVec{-1}{1} \in K$). 
\end{lemma}

\begin{lemma} \label{obs:adjacent}
Let $I,J \in \I_{q,2}$ be triangles such that $I\cap J \in \Ivert \cup \Ihor$.  Let $\pi$ be a continuous function defined on $I\cup J$ satisfying the following properties:
\begin{itemize}
\item[(i)] $\pi$ is affine on $I$.
\item[(ii)] There exists $c\in \R$ such that $\pi(\v + \lambda \ColVec{-1}{1}) = \pi(\v) + c\cdot\lambda$ for all $\v \in J$ and $\lambda \in \R$ such that $\v + \lambda \ColVec{-1}{1} \in J$.
\end{itemize}
Then $\pi$ is affine on $J$.
\end{lemma}

%
%
\section{Proof of the main results}

Let $\partial_\v$ denote the directional derivative in the direction of $\v$.   
\begin{definition}
  Let $\pi$ be a minimal valid function.
  \begin{enumerate}[\rm(a)]
  \item For any $I \in \P_q$, if $\pi$ is affine in
    $I$ and if for all valid functions $\pi^1, \pi^2$ such that $\pi =
    \tfrac{1}{2} \pi^1 + \tfrac{1}{2} \pi^2$ we have that $\pi^1, \pi^2$ are
    affine in $I$, then we say that $\pi$ is \emph{affine imposing in
      $I$}.
\item For any $I \in \P_q$, if $\partial_{(-1,1)}\pi$ is constant  in
    $I$ and if for all valid functions $\pi^1, \pi^2$ such that $\pi =
    \tfrac{1}{2} \pi^1 + \tfrac{1}{2} \pi^2$ we have that $\partial_{(-1,1)}\pi^1$, $\partial_{(-1,1)}\pi^2$ are
    constant in $I$, then we say that $\pi$ is \emph{diagonally affine imposing in
      $I$}.
  \item For a collection $\I\subseteq \P_q$, if for all
    $I \in \I$, $\pi$ is affine imposing (or diagonally affine imposing) in $I$, then we say that $\pi$ is
    \emph{affine imposing (diagonally affine imposing) in $\I$}.
  \end{enumerate}
\end{definition}

We either show that $\pi$ is affine imposing in $\P_q$
(subsection~\ref{section:AI}) or construct a continuous piecewise linear
$\Gamma$-equivariant perturbation over $\P_{4q}$ that proves
$\pi$ is not extreme (subsections \ref{sec:non-extreme-by-perturbation} and
\ref{sec:non-extreme-by-diag-perturbation}).  
If $\pi$ is affine imposing in $\P_q$, we set up a system of linear
equations to decide if~$\pi$ is extreme or not
(subsection~\ref{section:system}).   
This implies the main theorem stated in
the introduction.

\subsection{Imposing affine linearity on faces of~$\P_q$}
\label{section:AI}
For the remainder of this paper, we will use reflections and translations
modulo 1 to compensate for the fact that our function is periodic with period
$1$.  Working modulo $1$ is accounted for by applying the translations~$\tau_{(1,0)}$ and $\tau_{(0,1)}$ whenever needed.  Hence, we define the reflection ${\bar\rho}_\v(\x) = \v \ominus \x$ and the translation ${\bar\tau}_\v(\x) = \v \oplus \x$.  
The reflections and translations arise from certain valid triples as follows.
\begin{lemma} \label{lemma:translate}
Suppose $(I,J,K)$ is a valid triple.  
\begin{enumerate}
\item If $K= \{\a\} \in \I_{q,0}$, then $J = \bar \rho_\a(I)$,
\item If $J = \{\a\} \in \I_{q,0}$, then $K = \bar \tau_\a(I)$.
\end{enumerate}
\end{lemma}
\begin{proof}
\emph{Part 1.} Since $(I,J,\{\ve a\})$ is a valid triple, then for all $ \x \in I$, there exists a $\y \in J$ such that $x \oplus \y = \a$, i.e., $\y = \a \ominus \x \in J$, and therefore $J \supseteq \bar \rho_\a(I)$.  Also, for all $\y \in J$, there exists a $\x \in I$ such that $\x \oplus \y = \a$.  Again, $\y = \a \ominus \x$, i.e., $J \subseteq \bar \rho_\a(I)$.  Hence, $J = \bar\rho_\a(I)$.

\emph{Part 2.} Since $(I,\{\ve a\},K)$ is a valid triple and $J$ is a singleton, then for all $ \x \in I$, we have $\x \oplus \a \in K$, i.e., $K \supseteq \bar\tau_\a(I)$.  Also, for all $\ve z \in K$, there exists a $\x \in I$ such that $\x \oplus \a = \ve z$, i.e., $K \subseteq \bar\tau_\a(I)$.  Hence, $K =\bar\tau_\a(I)$.
\end{proof}

  Let $\G = \G(\I_{q,2},\E)$ be an  undirected graph with node set $\I_{q,2}$ and edge set $\E = \E_0 \cup \E_\diag$  where $\{I, J\} \in \E_0$ ($\{I, J\} \in \E_\diag$) if and only if  for some $K \in \I_{q,0}$ ($ K \in \Idiag$), we have $(I,J,K) \in E(\pi,\P_q)$ or $(I,K,J) \in E(\pi,\P_q)$. 
For each $I \in \I_{q,2}$, let $\G_I$ be the connected component of $\G$ containing $I$.  

We now consider faces of $\I_{q,2}$ on which we will apply lemmas from \autoref{s:real-analysis}.  
$$
\begin{array}{l}
\I^1_{q,2} = \{\,  I,J \in \P_{q,2} \st \exists K \in \Idiag \text{ with } (I,J,K) \in E(\pi,\P_q) \text{ or  } (I,K,J) \in E(\pi, \P_q) \,\},
\end{array}
$$
$$
\I^2_{q,2} = \{\, I,J,K\in \I_{q,2} \st (I,J,K) \in E(\pi, \P_q) \,\}.
$$
It follows from \autoref{cor:triangle+triangle} that $\pi$ is affine imposing in $\I^2_{q,2}$ and from \autoref{cor:triangle+diagonal} that $\pi$ is diagonally affine imposing in $\I^1_{q,2}$.  

Faces connected in the graph have related slopes.

\begin{lemma}
\label{lemma:point-and-line-Lemma}
Let $\v \in \R^2$.  For $\theta = \pi, \pi^1$, or $\pi^2$, 
if $\theta$ is affine in the $\v$ direction in $I$, i.e., there exists $c \in \R$  such that such that $\pi(\x + \lambda \v) = \pi(\x) + c\cdot\lambda$ for all $\x \in I$ and $\lambda \in \R$ such that $\x + \lambda \v \in I$, and $\{I,J\} \in \E$, then $\theta$ is affine in the $\v$ direction in $J$ as well.
%
\end{lemma}
The proof appears in appendix \ref{s:transferring-affine-linearity}.\smallbreak

With this in mind, we define the two sets of faces and any faces connected to them in the graph~$\G$,
$$
\S^1_{q,2} =\{\,J \in \I_{q,2} \st J \in \G_I \text{ for some } I \in \I^1_{q,2}\,\},
$$
$$
\S^2_{q,2} =\{\,J \in \I_{q,2} \st J \in \G_I \text{ for some } I \in \I^2_{q,2}\,\}.
$$
It follows from Lemma \ref{lemma:point-and-line-Lemma} that $\pi$ is affine imposing in $\S^2_{q,2}$ and diagonally affine imposing in $\S^1_{q,2}$.  

From \autoref{obs:adjacent}, it follows that if $I \in \S^2_{q,2}$, $J \in \S^1_{q,2}$ and $I\cap J \in \Ivert \cup \Idiag$, then $\pi$ is affine imposing in $J$.
Let 
$$\bar \S_{q,2} = \{ K \in \G_I \st I \in \S^1_{q,2} \text{ and there exists a } J \in \S^2_{q,2} \text{ such that } I \cap J \in \Ivert \cup \Ihor \}.$$  

Now set $\bar \S^2_{q,2} = \S^2_{q,2} \cup \bar \S_{q,2}$ and $\bar \S^1_{q,2} = \S^1_{q,2} \setminus \bar \S_{q,2}$.
The following theorem is a consequence of Lemmas \ref{cor:triangle+triangle},
\ref{obs:adjacent}, and \ref{lemma:point-and-line-Lemma}.
\begin{theorem}
\label{theorem:AI}
If $\bar \S^2_{q,2} = \I_{q,2}$, then $\pi$ is affine imposing in $\I_{q,2}$, and therefore $\theta$ is continuous piecewise linear over $\P_q$ for $\theta = \pi^1, \pi^2$.  
\end{theorem}


\subsection{Non-extremality by two-dimensional equivariant perturbation}
\label{sec:non-extreme-by-perturbation}

In this and the following subsection, we will prove the following result.

\begin{lemma}
  \label{lemma:not-extreme}
  Let $\pi$ be a minimal, continuous piecewise linear function over $\P_q$ that is diagonally constrained.  If $\bar \S^2_{q,2} \neq \I_{q,2}$, then $\pi$ is not extreme.
\end{lemma}

In the proof, we will need two different equivariant perturbations that we construct as
follows (see \autoref{s:equivariant}).  
Let $\Gamma_0 = \langle \,\rho_\g, \tau_\g \st \g \in \frac{1}{q}\Z^2\,\rangle$ be the group
generated by reflections and translations corresponding to all possible
vertices of $\P_q$.  We define the function $\psi\colon \R^2 \to \R$ as a
continuous piecewise linear function over $\P_{4q}$  in the following way: let
$T_0 = \frac1q \conv(\{ \begin{psmallmatrix}0\\0\end{psmallmatrix}
, \begin{psmallmatrix}1\\0\end{psmallmatrix}
, \begin{psmallmatrix}0\\1\end{psmallmatrix}
 \})$, and at all vertices
of $\P_{4q}$ that lie in~$T_0$, let $\psi$ take the value $0$, except at the
interior vertices $\frac{1}{4q}\begin{psmallmatrix}1\\1\end{psmallmatrix},
\frac{1}{4q}\begin{psmallmatrix}2\\1\end{psmallmatrix},
\frac{1}{4q}\begin{psmallmatrix}1\\2\end{psmallmatrix}$, where we assign
$\psi$ to have the value $1$.  Since $\psi$ is continuous piecewise linear
over $\P_{4q}$, this uniquely determines the function on $T_0$.   
We then extend $\psi$ to all of~$\R^2$ using the equivariance formula~\eqref{eq:equivariance}.
\begin{lemma}\label{lemma:our-equivariant-psi}
  The function $\psi\colon\R^2\to\R$ constructed above is well-defined and has the following properties:
  \begin{enumerate}[\rm(i)]
  \item $\psi(\g) = 0$ for all $\g\in \frac{1}{q}\Z^2$,
  \item $\psi(\x)= - \psi(\rho_{\g}(\x)) = - \psi(\g - \x)$ for all $\g \in \frac{1}{q}\Z^2, \x \in
    [0,1]^2$,
  \item $\psi(\x) = \psi(\tau_\g(\x)) = \psi(\g + \x)$ for all $\g \in \frac{1}{q}\Z^2, \x \in
    [0,1]^2$,
  \item $\psi$ is continuous piecewise linear over $\P_{4q}$.
  \end{enumerate}
\end{lemma}
\begin{proof}
  The properties follow directly from the equivariance formula~\eqref{eq:equivariance}.
\end{proof}



It is now convenient to introduce the function $\Delta\pi(\x,\y)= \pi(\x) + \pi(\y) - \pi(\x\oplus
\y)$, which measures the slack in the subadditivity constraints.
Let $\Delta\P_q$ be the polyhedral complex containing all polytopes $F=F(I,J,K)$ 
where $I,J,K \in\P_q$. 
Observe that $\Delta\pi|_{F}$ is affine; if we introduce the function $\Delta\pi_F(\x,\y) =
\pi_I(\x) + \pi_J(\y) - \pi_K(\x\oplus \y)$ for all $\x,\y\in\R^2$, then 
$\Delta\pi(\x,\y) = \Delta\pi_F(\x,\y)$ for all $(\x,\y) \in F$.  Furthermore, if $(I,J,K)$ is a valid triple, then $(I,J,K) \in E(\pi, \P_q)$ if and only if $\Delta \pi|_{F(I,J,K)}  = 0$.  We will use $\verts(F)$
to denote the set of vertices of the polytope~$F$. 

\begin{lemma}
\label{lemma:vertices}
Let  $ F \in \Delta \P_q$ and let $(\x,\y)$ be a vertex of $F$.  Then $\x,\y$ are vertices of the  complex $\P_q$, i.e., $\x,\y\in \frac{1}{q}\Z^2$.
\end{lemma}
The proof again uses the strong unimodularity properties of~$\P_q$ and appears
in appendix~\ref{s:complexes-unimodularity}.


\begin{lemma}\label{lemma:not-extreme1}
Let $\pi$ be a minimal, continuous piecewise linear function over $\P_q$ that is diagonally constrained. Suppose there exists $I^* \in \I_{q,2}\setminus ( \bar \S^2_{q,2} \cup \bar \S^1_{q,2})$.  Then $\pi$ is not extreme.
\end{lemma}
\begin{proof}
 Let $R = \bigcup_{J \in \G_{I^*}} \intr(J) \subseteq [0,1]^2$.   Since $R$ is a union of interiors, it does not contain any points in $\frac{1}{2q}\Z^2$.  
Let $\psi$ be the $\Gamma_0$-equivariant function of Lemma~\ref{lemma:our-equivariant-psi}.
Let 
$$
\epsilon = \min \{\,\Delta\pi_{\hat F}(\x,\y)\neq 0 \st {\hat F} \in \Delta \P_{4q}, \ (\x,\y) \in \verts(\hat F)\,\},
$$
and let $\bar \pi = \delta_R  \cdot\psi$ where $\delta_R$ is the indicator function for the set $R$. 
We will show that for 
$$\pi^1 = \pi + \tfrac{\epsilon}{3}\bar \pi, \  \ \pi^2 = \pi - \tfrac{\epsilon}{3}\bar \pi,$$
that $\pi^1, \pi^2$ are minimal, and therefore valid functions, and hence $\pi$ is not extreme.  We will show this just for $\pi^1$ as the proof for $\pi^2$ is the same.

Since $\psi(\0) =  0$ and $\psi(\f) = 0$, we see that $\pi^1(\0) =  0$ and $\pi^1(\f)= 1$.   

We want to show that $\pi^1$ is symmetric and subadditive. We will do this by analyzing the function $\Delta\pi^1(\x,\y) = \pi^1(\x) +  \pi^1(\y) - \pi^1(\x\oplus \y)$.   Since $\psi$ is piecewise linear over $\P_{4q}$, $\pi^1$ is also piecewise linear over $\P_{4q}$, and thus we only need to focus on vertices of $\Delta \P_{4q}$, which, by Lemma \ref{lemma:vertices}, are contained in $\frac1{4q} \Z^2$.

Let $\u,\v \in \frac1{4q} \Z^2$.  
First, if $\Delta\pi_(\u,\v) > 0$, 
then 
$$\Delta\pi^1(\u,\v) \geq \pi(\u) - \epsilon/3 + \pi(\v) - \epsilon/3 - \pi(\u\oplus \v) - \epsilon/3 = \Delta\pi(\u,\v) - \epsilon \geq 0.$$ %
 
Next, we will show that if $\Delta\pi(\u,\v) = 0$, then $\Delta\pi^1(\u,\v) = 0$.  This will prove two things.  First,   $\Delta\pi^1(\x,\y) \geq 0$ for all $\x,\y \in [0,1]^2$, and therefore $\pi^1$ is subadditive.  Second, since $\pi$ is symmetric,   $\Delta\pi(\x,\f \ominus \x) = 0$ for all $\x \in \frac1{4q}\Z^2$,  which would imply that $\Delta\pi^1(\x,\f \ominus \x) = 0$ for all $\x \in \frac{1}{4q} \Z^2$, proving $\pi^1$ is symmetric via \autoref{lemma:subSym}. 

Suppose that $\Delta\pi(\u,\v) = 0$.  We will proceed by cases.  \\
\textbf{Case 1.}  Suppose $\u,\v, \u\oplus \v \notin R$.  Then $\delta_R(\u) = \delta_R(\v) = \delta_R(\u\oplus \v) = 0$, and $\Delta\pi^1(\u,\v) = \Delta\pi(\u,\v) \geq 0$.\\
%
\noindent \textbf{Case 2.} Suppose we are not in Cases 1.  That is, suppose $\Delta\pi(\u,\v) = 0$, and at least one of $\u,\v,\u\oplus \v$ is in $R$.  Since $R \cap \frac{1}{2q}\Z^2 = \emptyset$, at least one of $\u,\v\u\oplus \v \notin \frac{1}{2q}\Z^2$.  This implies that at least one of $\u,\v \notin \frac{1}{2q}\Z^2$. 
Since $\Delta\pi^1(\x,\y)$ is symmetric in $\x$ and~$\y$, without loss of generality, we will assume that $\u \notin \tfrac{1}{2q} \Z^2$.  

Since $\u \notin \tfrac{1}{2q} \Z^2$, $(\u,\v) \notin \verts(\Delta\P_q)$.  Therefore, there exists a face $F \in \Delta\P_q$ such that $(\u,\v) \in \relint(F)$.  Since $\Delta \pi_F \geq 0$ ($\pi$ is subadditive) and $\Delta\pi_F(\u,\v) = 0$, it follows that $\Delta \pi_F = 0$.  Now let $(I,J,K) \in E_{\max{}}(\pi, \P_q)$ such that $F(I,J,K) \supseteq F$.
%
%
We discuss the possible cases for $I, J,K$ from Lemma \ref{lemma:cases}. 
\begin{enumerate}
\item If $I,J,K \notin \I_{q,2}$, then $I,J,K\in \I_{q,0} \cup \Idiag$ are all vertices or edges of $\P_q$, which are all not contained in $R$ since $R$ is the union of interiors of sets from $\I_{q,2}$.  Therefore, $\u, \v, \u\oplus \v \notin R$, which means we are in Case 1.  
\item If $I,J,K \in \I_{q,2}$, then $I,J,K \in \S^2_{q,2}$.  Therefore, $\u,\v,\u\oplus \v \notin R$, which means we are in Case 1.
\item One of $I,J,K$ is a diagonal edge in $\I_{q,1}$, while the other two are in $\I_{q,2}$, which means these sets are in $\S^1_{q,2}$.  Since edges are not in $R$, and $R \cap \S^1_{q,2} = \emptyset$, and again, $\u,\v,\u\oplus \v \notin R$, which means we are in Case 1.
\item This leaves us with the case where two of $I,J,K$ are in $\I_{q,2}$ and the third is a vertex, i.e.,  is in $\I_{q,0}$.  Since $\u \notin \frac1q\Z^2$, $I$ cannot be a vertex.  Therefore, $I \in \I_{q,2}$.  We proceed with this knowledge.
\end{enumerate}
%
%

%
There are two possible cases.   \\
 \indent \textbf{Case 2a.}  $J \in \I_{q,0}$, $I, K \in \I_{q,2}$ and hence $\v \in \frac{1}{q}\Z^2$.\\
Therefore $\{I, K\} \in
\E_0$ and $\delta_R(\u) = \delta_R(\u\oplus \v)$.  Since $\v \in \frac{1}{q}\Z^2$, we have
$\psi(\v) = 0$ and $\psi(\u) = \psi({\bar\tau}_\v(\u)) = \psi(\u\oplus \v)$ by
Lemma~\ref{lemma:our-equivariant-psi}~(iii).  It follows that $\bar \pi(\u) + \bar\pi(\v) - \bar\pi(\u\oplus \v) = 0$, and therefore $\Delta\pi^1(\u,\v) = \Delta\pi(\u,\v) = 0$.\\
\indent \textbf{Case 2b.}  $I, J \in \I_{q,2}$, $K \in \I_{q,0}$ and hence $\u \oplus \v \in \frac{1}{q}\Z^2$.
Therefore $\{I, J\} \in \E_0$ and $\delta_R(\u) = \delta_R(\v)$.  
 Since $\u\oplus \v \in \frac{1}{q}\Z^2$, $\psi(\u) = - \psi({\bar\rho}_{\u\oplus \v}(\u)) = - \psi(\v)$ by Lemma~\ref{lemma:our-equivariant-psi}~(ii).  It follows that $\bar \pi(\u) +  \bar\pi(\v) - \bar\pi(\u\oplus \v) = 0$, and therefore $\Delta\pi^1(\u,\v) = \Delta\pi(\u,\v) = 0$.\medskip

 We conclude that $\pi^1$ (and similarly $\pi^2$) is subadditive and symmetric, and therefore minimal and hence valid.  Therefore $\pi$ is not extreme.
\end{proof}

\subsection{Non-extremality by diagonal equivariant perturbation}
\label{sec:non-extreme-by-diag-perturbation}

We next construct a different equivariant perturbation function.  Let
$\Gamma_\smallsetminus = \langle \,\rho_\g, \tau_\g \st \ve1\cdot\g \equiv 0
\pmod{\frac{1}{q}}\,\rangle$, where $\ve 1=(1,1)$, be the group generated by reflections and
translations corresponding to all points on diagonal edges of $\P_q$.
We define the function $\varphi\colon \R^2 \to \R$ as a continuous piecewise linear function over $\P_{4q}$  in the following way: 
$$
\varphi(\x) = 
\begin{cases}
1 & \text{if } \ve1\cdot\x \equiv \frac{1}{4q} \pmod{\frac{1}{q}} ,\\
-1 & \text{if }\ve1\cdot\x \equiv \frac{3}{4q} \pmod{\frac{1}{q}},\\
0 &  \text{if }\ve1\cdot\x \equiv  0 \text{ or } \frac{2}{4q} \pmod{\frac{1}{q}}.
\end{cases}
$$
This function satisfies all properties of \autoref{lemma:our-equivariant-psi},
but is also $\Gamma_\smallsetminus$-equivariant.



\begin{lemma}\label{lemma:not-extreme2}
Suppose there exists $I^* \in \bar \S^1_{q,2}$ and $\pi$ is diagonally constrained.  Then $\pi$ is not extreme.
\end{lemma}
\begin{proof}
 Let $R = (\bigcup_{J \in \G_{I^*}} J) \setminus \{\, \x \st \ve1\cdot\x \equiv 0
 \text{ or $\frac2{4q}$} \pmod{\tfrac{1}{q}}\, \}$. \\

Let 
$$
\epsilon = \min \{\,\Delta\pi_{ F}(\x,\y)\neq 0 \st { F} \in \Delta \P_{4q}, \ (\x,\y) \in \verts( F)\,\},
$$
and let $\bar \pi$ be the unique continuous piecewise linear function over $\P_{4q}$ such that for any vertex $\x$ of $\P_{4q}$, we have $\bar\pi(\x) = \delta_{R}(\x)  \cdot\varphi(\x)$ where $\delta_R$ is the indicator function for the set $R$.  By construction, $\bar \pi$ is a continuous function that vanishes on all diagonal hyperplanes in the complex $\P_q$.
We will show that for 
$$\pi^1 = \pi + \tfrac{\epsilon}{3}\bar \pi, \  \ \pi^2 = \pi - \tfrac{\epsilon}{3}\bar \pi,$$
that $\pi^1, \pi^2$ are minimal, and therefore valid functions, and hence $\pi$ is not extreme.  We will show this just for $\pi^1$ as the proof for $\pi^2$ is the same.

Since, $\varphi(\0) =  0$ and $\varphi(\f) = 0$, we see that $\pi^1(\0) =  0$ and $\pi^1(\f)= 1$.   

We want to show that $\pi^1$ is symmetric and subadditive. We will do this by analyzing the function $\Delta\pi^1(\x,\y) = \pi^1(\x) +  \pi^1(\y) - \pi^1(\x\oplus \y)$.   Since $\bar \pi$ is continuous piecewise linear over $\P_{4q}$, $\pi^1$ is also continuous piecewise linear over $\P_{4q}$, and thus we only need to focus on vertices of $\Delta \P_{4q}$, which, by Lemma \ref{lemma:vertices}, are contained in $\frac1{4q} \Z^2$.

Let $\u,\v \in \frac1{4q} \Z^2$.   



First, if $\Delta\pi_(\u,\v) > 0$, then $\Delta\pi_(\u,\v) \geq \epsilon$ and 
 therefore $$\Delta\pi^1(\u,\v) \geq \pi(\u) - \epsilon/3 + \pi(\v) - \epsilon/3 - \pi(\u\oplus \v) - \epsilon/3 = \Delta\pi(\u,\v) - \epsilon \geq 0.$$ %
 
Next, we will show that if $\Delta\pi(\u,\v) = 0$, then $\Delta\pi^1(\u,\v) = 0$. 
This will prove two things.  First,   $\Delta\pi^1(\x,\y) \geq 0$ for all $\x,\y \in [0,1]^2$, and therefore $\pi^1$ is subadditive.  Second, since $\pi$ is symmetric,   $\Delta\pi(\x,\f \ominus \x) = 0$ for all $\x \in \frac1{4q}\Z^2$,  which would imply that $\Delta\pi^1(\x,\f \ominus \x) = 0$ for all $\x \in \frac{1}{4q} \Z^2$, proving $\pi^1$ is symmetric via \autoref{lemma:subSym}. 

Suppose that $\Delta\pi(\u,\v) = 0$.  We will proceed by cases.  \\
\textbf{Case 1.}  Suppose $\u,\v, \u\oplus \v \notin R$.  Then $\delta_R(\u) = \delta_R(\v) = \delta_R(\u\oplus \v) = 0$, and $\Delta\pi^1(\u,\v) = \Delta\pi(\u,\v) \geq 0$.\\
\textbf{Case 2.} Suppose $\u,\v \in \tfrac{1}{2q} \Z^2$.  Then $\ve 1 \cdot(\u\oplus \v) \equiv 0 \pmod{\tfrac1q}$ and, by definition of $R$, $\u,\v, \u\oplus \v \notin R$, and we are actually in Case 1.\\
\noindent \textbf{Case 3.} Suppose we are not in Cases 1 or 2.  That is, suppose $\Delta\pi(\u,\v) = 0$, not both $\u,\v$ are in $\smash[b]{\tfrac{1}{2q} \Z^2}$, and at least one of $\u,\v,\u\oplus \v$ is in $R$.
Since $\Delta\pi^1(\x,\y)$ is symmetric in $\x$ and~$\y$, without loss of generality, since not both $\u,\v$ are in $\tfrac{1}{2q} \Z^2$, we will assume that $\u \notin \tfrac{1}{2q} \Z^2$.  

Since $\u \notin \tfrac{1}{2q} \Z^2$, $(\u,\v) \notin \verts(\Delta\P_q)$.  Therefore, there exists a face $F \in \Delta\P_q$ such that $(\u,\v) \in \relint(F)$.  Since $\Delta \pi_F \geq 0$ ($\pi$ is subadditive) and $\Delta\pi_F(\u,\v) = 0$, it follows that $\Delta \pi_F = 0$.  Now let $(I,J,K) \in E_{\max{}}(\pi, \P_q)$ such that $F(I,J,K) \supseteq F$.
%
%
Since $\pi$ is diagonally constrained, by definition, $I,J,K$ are each either a vertex, diagonal edge, or triangle in $\P_q$.
We discuss the possible cases for $I, J,K$ according to Lemma \ref{lemma:cases}.  
\begin{enumerate}
\item If $I,J,K \notin \I_{q,2}$, then $I,J,K$ are all vertices or diagonal edges of $\P_q$, which are all not contained in $R$ since all vertices and diagonal edges are subsets of $ \{\,\x \st \x_1 + \x_2  \equiv 0 \pmod{\tfrac{1}{q}}\,\}$.  Therefore, $\u, \v, \u\oplus \v \notin R$, which means we are in Case 1.  
\item If $I,J,K \in \I_{q,2}$, then $I,J,K \in \S^2_{q,2}$.  By definition of $\bar \S^1_{q,2}$, for any $I' \in \S^2_{q,2}$ and $J' \in \bar \S^1_{q,2}$, either $I'\cap J' = \emptyset$, or $I' \cap J' \in \Idiag$.  Therefore, $\u,\v,\u\oplus \v \notin R$, which means we are in Case 1.  
\item If two of $I,J,K$ are in $\I_{q,2}$ and
  the third is a vertex, i.e.,  is in $\I_{q,0}$.  Since $\u \notin
  \frac1q\Z^2$, $I$ cannot be a vertex.  Therefore, $I \in \I_{q,2}$.  For
  this case, the proof is exactly the same as Case 2a and Case 2b in the proof
  of Lemma \ref{lemma:not-extreme1} because $\bar\pi(\x) = 0$ for all
  vertices~$\x\in\I_{q,0}$.  For brevity, we will not repeat it here.
\item If one of $I,J,K$ is in $\Idiag$, call it $I'$, 
 and the other two are in $\I_{q,2}$, call them $J'$,$K'$, then $J',K' \in \S^1_{q,2}$ and $\{J', K'\} \in \E_\diag$.  Since $I' \in \Idiag$, $I' \cap R = \emptyset$.
  Recall that $\S_{q,2}^1 \subseteq \bar\S_{q,2}^1 \cup \bar\S_{q,2}^2$.  If either $J'$ or $K'$ is in $\bar\S_{q,2}^2$, then they both are in $\bar\S_{q,2}^2$, i.e., $J' \cup K' \cap R = \emptyset$ and therefore $\u,\v,\u\oplus\v\notin R$, which is
  Case~1.  We proceed to consider the case where
$I' \in \Idiag$ and   $J',K'\in\bar\S_{q,2}^1$ with $\{J', K'\} \in \E_\diag$ of which there are three possible cases.
\end{enumerate}
%
%

%
 \indent \textbf{Case 3a.}   $ I\in \Idiag$, $J, K \in \I_{q,2}$.
Since $\{J,K\} \in \E_\diag$, $\delta_R(\v) = \delta_R(\u \oplus \v)$.  
Since $I\in\Idiag$ and $\u \in I$, $\ve 1\cdot \u \equiv 0 \pmod{\tfrac{1}{q}}$.  It follows that $\varphi(\u) = 0$ and  
$\ve1 \cdot \v \equiv \ve1 \cdot (\u \oplus \v) \pmod{\tfrac{1}{q}}$.  Therefore, $\varphi(\v) = \varphi(\u\oplus \v)$. Combining these, we have $\bar \pi(\u) + \bar\pi(\v) - \bar\pi(\u\oplus \v) = 0$, and therefore $\Delta\pi^1(\u,\v) = \Delta\pi(\u,\v) = 0$.\\
  \indent \textbf{Case 3b.}   $ J\in \Idiag$, $I, K \in \I_{q,2}$.
This is similar to Case 3a and the proof need not be repeated.\\
\indent \textbf{Case 3c.}  $I, J \in \I_{q,2}$, $K \in \Idiag$ and hence
$\ve1\cdot(\u \oplus \v) \equiv 0 \pmod{\frac1q}$.
Since $\{I,J\} \in \E_\diag$, we have $\delta_R(\u) = \delta_R(\v)$.   
 Since $\ve1\cdot(\u\oplus \v) \equiv 0 \pmod{\frac1q}$, we have
 $\ve1\cdot\u\equiv -\ve1\cdot\v \pmod{\frac1q}$, and hence $\varphi(\u) = -
 \varphi(\v)$.  It follows that $\bar \pi(\u) +  \bar\pi(\v) -
 \bar\pi(\u\oplus \v) = 0$, and therefore $\Delta\pi^1(\u,\v) =
 \Delta\pi(\u,\v) = 0$.\medskip

 We conclude that $\pi^1$ (and similarly $\pi^2$) is subadditive and symmetric, and therefore minimal and hence valid.  Therefore $\pi$ is not extreme.
\end{proof}

\begin{proof}[of Lemma~\ref{lemma:not-extreme}]
This follows directly from Lemmas \ref{lemma:not-extreme1} and \ref{lemma:not-extreme2}.
\end{proof}

The specific form of our perturbations as continuous piecewise linear
function over $\P_{4q}$ implies the following corollary.

\begin{corollary}
\label{corollary:AIG4}
Suppose $\pi$ is a continuous piecewise linear function over $\P_q$ and is diagonally constrained.  If $\pi$ is not affine imposing over $\I_{q,2}$, then there exist distinct minimal
$\pi^1, \pi^2$ that are continuous piecewise linear over $\P_{4q}$ such that $\pi = \tfrac{1}{2}\pi^1+ \tfrac{1}{2} \pi^2$.  
\end{corollary}

\subsection{Extremality and non-extremality by linear algebra}
\label{section:system}

In this section we suppose $\pi$ is a minimal  continuous piecewise linear function over $\P_{q}$ that is affine imposing in  $\I_{q,2}$.  Therefore, $\pi^1$ and $\pi^2$ must also be continuous piecewise linear functions over $\P_{q}$.   It is clear that whenever $\pi(\x) + \pi(\y) = \pi(\x\oplus \y)$, the functions $\pi^1$ and $\pi^2$ must also satisfy this equality relation, that is, $\pi^i(\x) + \pi^i(\y) = \pi^i(\x\oplus \y)$.  
%
We now set up a system of linear equations that $\pi$ satisfies and that
$\pi_1$ and $\pi_2$ must also satisfy.
Let $\varphi\colon \frac{1}{q}\Z^2 \to \R$.  
 Suppose $\varphi$ satisfies the following system of linear equations:
\begin{equation}
\label{equation:system}
\begin{cases}
\varphi(\0) = 0,\ \varphi(\f) = 1,\ \varphi(\ColVec{0}{1}) = 0,\
\varphi(\ColVec{0}{1}) = 0,\  \varphi(\ColVec{1}{1})) = 0,\\
\varphi(\u) + \varphi(\v)  =  \varphi(\u \oplus \v) 
\text{ if $\u,\v \in \frac{1}{q} \Z^2$, } 
\pi(\u) + \pi(\v)  =  \pi(\u \oplus \v)
\end{cases}
\end{equation}

 Since $\pi$ exists and satisfies~\eqref{equation:system}, we know that the system has a solution.
\begin{theorem}
\label{theorem:systemNotUnique}
Let $\pi\colon \R^2 \to \R$ be a continuous piecewise linear valid function
over $\P_q$.
\begin{enumerate}[\rm i.]
\item If the system \eqref{equation:system} does not have a
  unique solution, then $\pi$ is not extreme. 
\item Suppose $\pi$ is minimal and affine imposing in $\I_{q,2}$.
  Then $\pi$ is extreme if and only if the system of equations
  \eqref{equation:system} has a unique solution.  
\end{enumerate}
\end{theorem}

The proof, similar to one in~\cite{basu-hildebrand-koeppe:equivariant},
appears in appendix~\ref{s:proof-theorem:systemNotUnique}. 

\subsection{Connection to a finite group problem}

\begin{theorem}\label{thm:1/4q test}
Let $\pi$ be a minimal continuous piecewise linear function over $\P_q$ that is diagonally constrained.  
Then $\pi$ is extreme if and only if the system of
equations~\eqref{equation:system} with $\tfrac{1}{4q}\Z^2$
has a unique solution.
\end{theorem}
\begin{proof}
  Since $\pi$ is continuous piecewise linear over $\P_q$, it is also continuous piecewise linear over $\P_{4q}$.  The forward direction is the contrapositive of
  Theorem~\ref{theorem:systemNotUnique}\,(i), applied when we view $\pi$ piecewise linear over $\P_{4q}$. 
  For the reverse direction, observe that if the system of
  equations~\eqref{equation:system} with $\frac1{4q}\Z^2$ has a unique
  solution, then there cannot exist distinct minimal $\pi^1, \pi^2$ that are
  continuous piecewise linear over~$\P_{4q}$ such that $\pi = \frac12\pi^1 +
  \frac12\pi^2$.  By the contrapositive of Corollary \ref{corollary:AIG4},
  $\pi$ is affine imposing in $\I_{q,2}$.  
  Then $\pi$ is also affine imposing on $\I_{4q,2}$ since it is a finer set.  
  By \autoref{theorem:systemNotUnique}\,(ii), since $\pi$ is affine imposing in $\I_{4q,2}$ and the system of
  equations~\eqref{equation:system} on $\P_{4q}$ has a unique solution, $\pi$ is extreme.  
\end{proof}

Theorem \ref{thm:main} is proved by testing for minimality using Lemma \ref{minimality-check} and then testing for extremality using Theorem \ref{thm:1/4q test}.   Theorem \ref{thm:1/4q} is a direct consequence  of Theorem \ref{thm:1/4q test}.
\bibliographystyle{abbrv}
\bibliography{../bib/MLFCB_bib}

\begin{thebibliography}{10}

\bibitem{bccz08222222}
Amitabh Basu, Michele Conforti, G{\'e}rard Cornu{\'e}jols, and Giacomo
  Zambelli.
\newblock A counterexample to a conjecture of {G}omory and {J}ohnson.
\newblock {\em Mathematical Programming Ser. A}, 133:25--38, 2012.

\bibitem{basu-hildebrand-koeppe:equivariant}
Amitabh Basu, Robert Hildebrand, and Matthias K\"oppe.
\newblock Equivariant perturbation in {G}omory and {J}ohnson's infinite group
  problem.
\newblock eprint arXiv:{\penalty0}1206.2079 [math.OC], 2012.
\newblock manuscript.

\bibitem{bhkm}
Amitabh Basu, Robert Hildebrand, Matthias K\"oppe, and Marco Molinaro.
\newblock A~$(k+1)$-slope theorem for the $k$-dimensional infinite group
  relaxation.
\newblock eprint arXiv:{\penalty0}1109.4184 [math.OC], 2011.

\bibitem{3slope}
G{\'e}rard Cornu{\'e}jols and Marco Molinaro.
\newblock A 3-slope theorem for the infinite relaxation in the plane.
\newblock {\em Mathematical Programming}, pages 1--23, 2012.

\bibitem{deyRichard}
Santanu~S. Dey and Jean-Philippe~P. Richard.
\newblock Facets of two-dimensional infinite group problems.
\newblock {\em Mathematics of Operations Research}, 33(1):140--166, 2008.

\bibitem{dey2}
Santanu~S. Dey and Jean-Philippe~P. Richard.
\newblock Relations between facets of low- and high-dimensional group problems.
\newblock {\em Math. Program.}, 123(2):285--313, June 2010.

\bibitem{dey1}
Santanu~S. Dey, Jean-Philippe~P. Richard, Yanjun Li, and Lisa~A. Miller.
\newblock On the extreme inequalities of infinite group problems.
\newblock {\em Math. Program.}, 121(1):145--170, June 2009.

\bibitem{gom}
Ralph~E. Gomory.
\newblock Some polyhedra related to combinatorial problems.
\newblock {\em Linear Algebra and its Applications}, 2(4):451--558, 1969.

\bibitem{infinite}
Ralph~E. Gomory and Ellis~L. Johnson.
\newblock Some continuous functions related to corner polyhedra, {I}.
\newblock {\em Mathematical Programming}, 3:23--85, 1972.
\newblock 10.1007/BF01585008.

\bibitem{infinite2}
Ralph~E. Gomory and Ellis~L. Johnson.
\newblock Some continuous functions related to corner polyhedra, {II}.
\newblock {\em Mathematical Programming}, 3:359--389, 1972.
\newblock 10.1007/BF01585008.

\bibitem{tspace}
Ralph~E. Gomory and Ellis~L. Johnson.
\newblock T-space and cutting planes.
\newblock {\em Mathematical Programming}, 96:341--375, 2003.
\newblock 10.1007/s10107-003-0389-3.

\end{thebibliography}

\clearpage
\appendix
\section{Appendix}

\subsection{Equivariant perturbations} 
\label{s:equivariant}

In this section we outline the theory of equivariant perturbations for the
infinite group problem, used first in
\cite{basu-hildebrand-koeppe:equivariant} for the case $k=1$.
  
We consider a subgroup of the group $\Aff(\R^k)$ of invertible affine linear
transformations of~$\R^k$ as follows. 
\begin{definition}
  For a point $\ve r \in \R^k,$ define the \emph{reflection} $\rho_{\ve
    r}\colon \R^k\to \R^k$, $\ve x
  \mapsto \ve r-\ve x$.  For a vector $t \in \R^k,$ define the \emph{translation}
  $\tau_{\ve t} \colon \R^k\to \R$, $\ve x \mapsto \ve x + \ve t$.  
\end{definition}
Given a set $R$ of points and a set $U$ of vectors, we will
define the subgroup 
$$\Gamma= \langle\, \rho_{\ve r}, \tau_{\ve t} \st \ve r\in R,\, \ve t\in
U\,\rangle.$$  Let $\ve r,\ve s,\ve w,\ve t\in \R^k$.  Each reflection is an
involution: $\rho_{\ve r} \circ \rho_{\ve r} =\textit{id}$,  two reflections give one
translation: $\rho_{\ve r} \circ \rho_{\ve s} = 
\tau_{\ve r-\ve s}$.  Thus, if we assign a \emph{character} $\chi(\rho_{\ve r}) = -1$ to
every reflection and $\chi(\tau_{\ve t}) = +1$ to every translation, then this
extends to a \emph{group character} of~$\Gamma$, that is, a group homomorphism
$\chi\colon \Gamma\to\C^\times$.  

On the other hand, not all pairs of reflections need to be
considered: $\rho_{\ve s}\circ \rho_{\ve w} = (\rho_{\ve s}\circ \rho_{\ve r})
\circ (\rho_{\ve r}
\circ \rho_{\ve w}) = (\rho_{\ve r}\circ \rho_{\ve s})^{-1} \circ (\rho_{\ve
  r} \circ \rho_{\ve w})$.
Thus the subgroup~$T = \ker \chi$ of translations in~$\Gamma$ is generated as
follows. Let $\ve r_1 \in R$ be any of the reflection points; then 
$$T = 
\langle\, \tau_{\ve r-\ve r_1}, \tau_{\ve t} \st \ve r\in R,\, \ve t\in U\,
\rangle.$$  
It is \emph{normal} in~$\Gamma$, as it is stable by conjugation by any
reflection: $\rho_{\ve r} \circ \tau_{\ve t} \circ \rho_{\ve r}^{-1} =
\tau_{-\ve t}$.  If $\gamma \in \Gamma$ is
not a translation, i.e., $\chi(\gamma) = -1$, then it is generated by an odd
number of reflections, and 
thus can be written as $\gamma = \tau \rho_{\ve r_1}$ with $\tau\in T$.  Thus $\Gamma /
T = \langle\rho_{\ve r_1}\rangle$ is of order~$2$. In short, we have the following
lemma. 
\begin{lemma} \label{lemma:semidirect} The group $\Gamma$ is the semidirect
  product $T \rtimes \langle \rho_{\ve r_1} \rangle$, where the (normal)
  subgroup of translations can be written as $$T = \{\, \tau_{\ve t} \st \ve t
  \in \Lambda\,\},$$ where $\Lambda$ is the additive subgroup $$\Lambda =
  {\langle\, \ve r-\ve r_1, \ve t \st \ve r\in R,\, \ve t\in U\, \rangle}_\Z \subseteq
  \R^k.$$%
\end{lemma}%
\begin{definition}
  A function $\psi\colon\R^k\to\R$ is called \emph{$\Gamma$-equivariant} if 
  it satisfies the \emph{equivariance formula} 
  \begin{equation}\label{eq:equivariance}
    \psi(\gamma(\x)) = \chi(\gamma) \psi(\x)
    \quad\text{for $\x\in \R$ and $\gamma \in \Gamma$}.
  \end{equation}  
\end{definition}
We note that if $\Lambda$ is discrete, i.e., a lattice, then there is a way to
construct continuous $\Gamma$-equivariant functions by defining them on a
fundamental domain and extending them to all of~$\R^k$ via the equivariance
formula~\eqref{eq:equivariance}.  The same is true for the case where
$\Lambda$ is a mixed lattice, i.e., a direct sum of a lattice in a
subspace and another subspace.  We omit the details. 

\subsection{Polyhedral complexes $\P_q$, $\Delta\P_q$, and unimodularity}
\label{s:complexes-unimodularity}

We first comment that $\f$ must be a vertex of $\P_q$ of any minimal
valid function.  We omit the proof here as it is very similar to (\cite{basu-hildebrand-koeppe:equivariant}, Lemma 2.1).
\begin{lemma}
\label{lemma:faVertex}
If $\pi$ is a minimal function, then $\f \in \frac{1}{q}\Z^2$.  
\end{lemma}
%

\begin{definition}\label{def:valid-triple}
For $I,J,K \in \P_q\setminus \{\emptyset\}$, we say $(I,J,K)$ is a \emph{valid triple} provided that the following occur:
\begin{enumerate}[i.]
\item $K \subseteq I \oplus J$,
\item For all $\u \in I$ there exists a $\v \in J$ such that $\u\oplus \v \in K$,
\item For all $\v \in J$ there exists a $\u \in I$ such that $\u \oplus \v \in K$,
\end{enumerate}
\end{definition}

Equivalently, a valid triple $(I,J,K)$ is characterized by the following property.
\begin{enumerate}
\item[iv.] Whenever $I', J', K'$ are sets
  such that $I'\subseteq I$, $J' \subseteq J$, $K' \subseteq K$ and
  $F(I,J,K) = F(I', J', K')$ we have that $I' = I$, $J' = J$, $K' = K$.
\end{enumerate}

The construction of $\P_q$ has convenient properties such as the following.
\begin{lemma}\label{lemma:I+J}
Let $I, J \in \P_q$.  Then $I \oplus J$ and $I \ominus J$ are both unions of faces in $\P_q$.
\end{lemma}
\begin{proof}
By construction, for any face $K \in \P_q$, the set $\{\,\x \bmod{\ve1} \st \x
\in K\,\}$ is also a face in $\P_q$.  Therefore we only need to show that
the Minkowski sums $I+J$ and $I-J$ are unions of faces in $\P_q$. 
Let 
$$
A =  \begin{bmatrix}  1 & -1 &  0 & 0 &  1 & -1 &\\
 0 & 0 & 1 & -1  & 1 & -1  \end{bmatrix}^T
$$
Let $\ve a^i$ be the $i^\text{th}$ row vector of $A$. 
Then there exists vectors $\ve b^1, \ve b^2$ such that $I = \{\, \x \st A \x
\leq \ve b^1\,\}$, $J = \{\, \y \st A \y \leq \ve b^2\,\}$.
Moreover, due to the total unimodularity of the matrix~$A$, the right-hand
side vectors
$\ve b^1, \ve b^2$ can be chosen so that $\ve b^1, \ve b^2$ are tight, i.e., 
\begin{equation}
\max_{\x \in I} \ve a^i \cdot \x = \ve b^1_i, \quad    \max_{\y \in J}  \ve a^i \cdot \y = \ve b^2_i,
\label{eq:support-vector}
\end{equation}
and $\ve b^1, \ve b^2 \in \frac1q\Z^2$.

We claim that $I + J = \{\, \x \st A \x \leq \ve b^1 + \ve b^2 \,\}$.  Clearly
$I+J \subseteq \{\, \x \st A \x \leq \ve b^1 + \ve b^2 \,\}$.  We show the
reverse direction.    
Let $K'$ be a facet (edge) of $I+J$.  Then $K'=I'+J'$, where $I'$ is a face of
$I$ and $J'$ is a face of $J$.  Without loss of generality, assume that $I'$
is an edge; then $J'$ is either a vertex or an edge.  
By well-known properties of Minkowski sums, the normal cone of $K'$ is the
intersection of the normal cones of $I'$ in $I$ 
and $J'$ in $J$.  Thus $K'$ has the same normal direction as the facet
(edge)~$I'$.  This proves that $I + J = \{\, \x \st A \x \leq \ve b\,\}$ for
some vector $\ve b$. 

Let $\x^*$, $\ve y^*$ be maximizers in~\ref{eq:support-vector}.  Then $\x^* + \y^* \in I + J$.  Then 
$$
\ve b^1_i  + \ve b^2_i  = \ve a^i \cdot \x^* +  \ve a^i \cdot \y^* \leq  \max_{\ve z \in I + J}  \ve a^i \cdot \ve z \leq \max_{\x \in I} \ve a^i \cdot \x  + \max_{\y \in J}  \ve a^i \cdot \y = \ve b^1_i + \ve b^2_i.
$$
Therefore, $\max_{\ve z \in I + J}  \ve a^i \cdot \ve z = \ve b^1_i + \ve
b^2_i$,  which shows that every constraint $a_i \cdot \ve z \leq \ve b^1_i$ is
met at equality, and therefore $I+J = \{\, \x \st A \x \leq \ve b^1 + \ve b^2\,\}$
and we conclude that $I + J$ must be a union of subsets in $\P_q$. 

The case $I - J = \{\,\ve z - \y \st \ve z \in K, \y \in J\,\}$ is shown similarly.
\end{proof}

\begin{lemma}\label{lemma:covered-by-maximal-valid-triples}
$
E(\pi) = \bigcup \{ F(I,J,K) \st (I,J,K) \in E_{\max{}}(\pi, \P_q)  \}.
$
\end{lemma}

\begin{proof}
Clearly the right hand side is a subset of $E(\pi)$.  We show $E(\pi)$ is a subset of the right hand side.
Suppose $(\x,\y) \in E(\pi)$.  Let $I, J, K \in \P_q$ be minimal faces by set inclusion containing $\x$, $\y$, and $\x \oplus \y$, respectively.  We show that $(I,J,K)$ is a valid triple.  
By Lemma \ref{lemma:I+J}, $I \oplus J$ is a union of faces in $\P_q$.  Since $\x \oplus \y \in I\oplus J$ and $\x \oplus \y \in K$, we have that $K \cap (I+J) \neq \emptyset$, and in particular, is a union of faces of $\P_q$ containing $\x + \y$.  Since $K$ was chosen to be a minimal such face in $\P_q$ containing $\x\oplus\y$, we have that $K \subseteq I\oplus J$.  \\

Similarly, by Lemma \ref{lemma:I+J},  $K \ominus J$ is also a union of sets in $\P_q$ containing~$\x$.  Since $I$ is a minimal set containing $\x$, it must be that $I \subseteq K \ominus J$.  Therefore, for any $\u \in I$, there exists a $\v \in J$ such that $\u\oplus\v \in K$. \\
Similarly, we find that for any $\v \in J$, there exists a $\u \in I$ such that $\u\oplus \v \in K$.

Since $I, J, K$ were chosen to be minimal in $\P_q$, the triple satisfies criterion (iv) of being a valid triple.  Hence, $(I,J,K)$ is a valid triple.

Next we argue that $(I,J,K) \in E(\pi, \P_q)$.
This is because $\Delta \pi$ is affine in $F(I,J,K)$, $\Delta \pi \geq 0$, $(\x,\y) \in \relint(F(I,J,K))$, $\Delta\pi(\x,\y) = 0$ and therefore $\Delta\pi|_{F(I,J,K)} = 0$, i.e., $(I,J,K) \in E(\pi, \P_q)$.

Lastly, if $(I,J,K)$ is not maximal in $E(\pi, \P_q)$, then there exists a maximal $(I', J', K')$ such that $F(I', J', K') \supset (I,J,K)$, namely, $(\x,\y) \in F(I', J', K')$.
\end{proof}

Next we study the complex $\Delta \P_q$. 

\begin{proof}[of \autoref{lemma:vertices}]
Since $F \in \Delta \P_q$, we can write $F$ using the system of inequalities $F = \{(\x,\y) \in \R^4 \colon \hat A (\x,\y) \leq \ve b\}$  where $\ve b \in \frac{1}{q} \Z^9$, the matrix $A$ is given by 
$$
A = 
\begin{bmatrix}
 1 & 0 & 1 & 0 & 0 & 0 & 1 & 0 & 1\\ 0 & 1 & 1 & 0 & 0 & 0 & 0 & 1 & 1\\ 0 & 0 & 0 & 1 & 0 & 1 & 1 & 0 & 1\\ 0 & 0 & 0 & 0 & 1 & 1 & 0 & 1 & 1 
\end{bmatrix}^T
%
$$
and the matrix $\hat A$ differs from $A$ only by scaling each row individually
by $\pm 1$.  (This inequality representation of~$F$ will usually be redundant.)
By checking every subdeterminant of the matrix $A$, it can be verified that $A$ is totally unimodular, and therefore $\hat A$ is also totally unimodular.   Therefore, the polytope 
$qF = \{\,(\x,\y) \in \R^4 \colon \hat A (\x,\y) \leq q \ve b\,\}$ has integral vertices in $\Z^4$.  

It follows that $P$ has vertices in $\frac{1}{q} \Z^4$.  Therefore, $\x,\y \in \frac{1}{q} \Z^2$ and therefore are vertices of $\P_q$.
\end{proof}

\subsection{Continuity results}
\label{s:continuity}

In this section we prove \autoref{Theorem:functionContinuous} on continuity.  Although similar results appear in~\cite{infinite2}, we provide proofs of these facts to keep this paper more self-contained.  
We first prove the following lemma.  

\begin{lemma}\label{lem:lipschitz}
If $\theta\colon \R^k \to \R$ is a subadditive function and $\limsup_{h\to 0}
\frac{\lvert\theta(\ve h)\rvert}{\lvert \ve h\rvert} = L < \infty$, then $\theta$ is Lipschitz continuous with Lipschitz constant $L$.
\end{lemma}
\begin{proof}
  Fix any $\delta > 0$. Since $\limsup_{{\ve h}\to 0} \frac{\lvert\theta({\ve
      h})\rvert}{\lvert {\ve h}\rvert} = L$, there exists $\epsilon > 0$ such
  that for any ${\ve x},{\ve y} \in \R^k$ satisfying $|{\ve x} - {\ve y}| <
  \epsilon$, $\frac{|\theta({\ve x}-{\ve y})|}{|{\ve x}-{\ve y}|} <
  L+\delta$. By subadditivity, $|\theta({\ve x}-{\ve y})| \geq |\theta({\ve
    x}) - \theta({\ve y})|$ and so $\frac{|\theta({\ve x}) - \theta({\ve
      y})|}{|{\ve x}-{\ve y}|} < L + \delta$ for all ${\ve x},{\ve y} \in
  \R^k$ satisfying $|{\ve x} - {\ve y}| < \epsilon$. This immediately implies
  that for {\em all} ${\ve x},{\ve y} \in \R$, $\frac{|\theta({\ve x}) -
    \theta({\ve y})|}{|{\ve x}-{\ve y}|} < L + \delta$, by simply breaking the
  interval $[{\ve x},{\ve y}]$ into equal subintervals of size at most
  $\epsilon$. Since the choice of $\delta$ was arbitrary, this shows that for
  every $\delta > 0$, $\frac{|\theta({\ve x}) - \theta({\ve y})|}{|{\ve
      x}-{\ve y}|} < L + \delta$ and therefore, $\frac{|\theta({\ve x}) -
    \theta({\ve y})|}{|{\ve x}-{\ve y}|} \leq L$. Therefore, $\theta$ is
  Lipschitz continuous with Lipschitz constant $L$.
\end{proof}

\begin{proof}[of Theorem~\ref{Theorem:functionContinuous}]
The minimality of $\pi^1, \pi^2$ is clear. Since we assume $\pi^1, \pi^2 \geq 0$, $\pi = \frac{1}{2}\pi^1 + \frac{1}{2}\pi^2$ implies that $\pi^i \leq 2\pi$ for $i = 1,2$. Therefore if $\limsup_{{\ve h}\to 0} \frac{\lvert\pi({\ve h})\rvert}{\lvert {\ve h}\rvert} = L < \infty$, then $\limsup_{{\ve h}\to 0} \frac{\lvert\pi^i({\ve h})\rvert}{\lvert {\ve h}\rvert} \leq 2L< \infty$ for $i=1,2$.
Applying Lemma~\ref{lem:lipschitz}, we get Lipschitz continuity for all three functions.\end{proof}

The following is a slight generalization of the Interval Lemma that appears in~\cite{bccz08222222}. The proof is a minor modification of the original proof.

\subsection{Finite test for minimality of piecewise linear functions}
\label{section:minimalityTest}
In this subsection, we show that there is an easy test to see if a continuous piecewise
linear function over $\P_q$ is minimal.  

\begin{lemma}
\label{lemma:subSym}
Suppose that $\pi$ is a continuous piecewise
linear function over $\P_q$ and $\pi(\0) = 0$. 
 \begin{enumerate}
\item
$\pi$  is subadditive if and only if  $\pi(\x) + \pi(\y) \geq \pi(\x\oplus \y)$ for all $\x,\y \in \frac{1}{q}\Z^2$,

\item $\pi$ is symmetric if and only if $\pi(\x) + \pi(\f\ominus \x) = 1$ for all $\x \in \frac{1}{q} \Z^2$.
\end{enumerate}
 \end{lemma}

 \begin{proof}
Clearly the forward direction of both statements is true.  We will show the reverse of each.  
For subadditivity, we need to show that $\Delta \pi \geq 0$.  Since $\Delta  \pi$ is piecewise linear over $\Delta \P_q$, we just need to show that $\Delta \pi(\x,\y) \geq 0$ for any $(\x,\y) \in \verts(\Delta \P_q)$.  By Lemma \ref{lemma:vertices}, $\verts(\Delta \P_q) \subseteq  \frac{1}{q}\Z^4$, and the result follows.\\
Next, we show symmetry.  Since $\0,\f \in \frac{1}{q}\Z^2$ and $\pi(\0) = 0$, we have that $\pi(\f) = 1$.  Let $\x \in [0,1]^2$ and let $F \in \Delta \P_q$ such that $(\x, \f \ominus \x) \in F$.

 Similarly, to show symmetry, we need to show that $\Delta\pi(\x,\y) = 0$ for all $\x,\y \in [0,1]^2$ such that $\x \oplus  \y = \f$.  Let $\x,\y \in [0,1]^2$ such that $\x\oplus \y = \f$.  Since $\f \in \frac{1}{q}\Z^2$ by Lemma \ref{lemma:faVertex}, $(\x,\y) \in \relint(\hat F)$ for some face $\hat F$ of some $F \in \Delta \P_q$ and  $\hat F \subseteq \{\,(\x,\y) \st \x\oplus \y = f\,\}$.   Since $\Delta\pi_F(\u,\v) = 0$ for all $(\u,\v) \in \verts(F) \subset \frac{1}{q}\Z^2$ when $\u \oplus \v = \f$, and $\Delta\pi_F$ is affine, it follows that $\Delta\pi(\x,\y) = \Delta\pi_F(\x,\y) = 0$.  
\end{proof}
The following theorem is a direct corollary of Lemma \ref{lemma:subSym} and Theorem \ref{thm:minimal}.
\begin{theorem}[Minimality test]
\label{minimality-check}
A function $\pi \colon \R^2 \to \R$ that is continuous piecewise linear over $\P_q$ is minimal if
and only if
\begin{enumerate}
\item $\pi(\0) = 0$,
\item $\pi(\x) + \pi(\y) \geq \pi(\x\oplus \y)$ for all $\x,\y \in \frac{1}{q}\Z^2$,
\item  $\pi(\x) + \pi(\f\ominus \x) = 1$ for all $\x \in \frac{1}{q} \Z^2$.
\end{enumerate}
\end{theorem}

\subsection{Properties of valid triples}
\label{s:properties-valid-triples}

\begin{lemma}\label{lemma:cases}
Suppose $\pi$ is continuous piecewise linear over $\P_q$ and is diagonally constrained.  Suppose that $(I,J,K) \in E(\pi, \P_q)$.  Then one of the following is true.
\begin{enumerate}
\item $I,J,K \in \P_{q,0} \cup \Idiag$,
\item  $I, J,K \in \P_{q,2}$,
\item One of $I,J,K$ is in $\I_{q,0}$, while the other two are in $\I_{q,2}$,
\item One of $I,J,K$ is in $\Idiag$, while the other two are in $\I_{q,2}$
\end{enumerate}
\end{lemma}
\begin{proof}
By definition of diagonally constrained, $I,J,K \in \I_{q,0} \cup \Idiag \cup \I_{q,2}$.  There are 27 possible ways to put $I,J,K$ into those three sets.  Above, 15 possibilities are described.  We will show that the 12 remaining cases not list above are not possible because $(I,J,K)$ is assumed to be a valid triple.
\begin{enumerate}
\item Suppose $I, J \in \I_{q,0} \cup \Idiag$, $K \in I_{q,2}$.   Then $K' = I \oplus J \subsetneq K$, and therefore $F(I,J,K) = F(I,J,K')$, and therefore $(I,J,K)$ is not a valid triple.
\item Suppose $I, K\in \I_{q,0} \cup \Idiag$, $J \in I_{q,2}$.  Then $K \ominus I \subsetneq J$, and therefore, there exists a $J' \subsetneq J$ such that $F(I,J,K) = F(I,J',K)$, and therefore $(I,J,K)$ is not a valid triple.
\item Suppose $J, K\in \I_{q,0} \cup \Idiag$, $I \in I_{q,2}$. This is similar to the last case.
\end{enumerate}
\end{proof}

\begin{lemma}\label{lemma:interiorDiag}
Suppose $(I,J,K)$ is a valid triple. The following are true.
\begin{enumerate}[\rm i.]
\item Suppose $I,J \in \I_{q,2}$. Then for every point $\u \in \intr(I)$ there exists a point $\v \in \intr(J)$ such that $\u \oplus \v \in \relint(K)$.
\item Suppose $I,K \in \I_{q,2}$. Then for every point $\ve w \in \intr(K)$ there exists a point $\u \in \intr(I)$ such that $\ve w \ominus \u \in \relint(J)$.
\end{enumerate}
\end{lemma}

\begin{proof}
\emph{Part (i).}
Let $\u \in \intr(I)$ and so $(1,0)^T\u$, $(0,1)^T\u$ and $(1,1)^T\u$ are all nonzero modulo $\frac1q$. Since $(I,J,K)$ is a valid triple, there exist $\v \in J$ and $\ve w \in K$ such that $\u \oplus \v = \ve w$. Thus, $(1,0)^T\v$ and $(1,0)^T\ve w$ are different modulo $\frac1q$ (resp. for $(0,1)^T\v$, $(0,1)^T\ve w$ and $(1,1)^T\v$, $(1,1)^T\ve w$). Note that for any point $\x \in \R^2$, either $(1,0)^T\ve x$, $(0,1)^T\ve x$ and $(1,1)^T\ve x$ are all $0$ modulo $\frac1q$, or exactly one of these numbers is $0$ modulo $\frac1q$, or none of them are $0$. Thus, we consider these cases :

{\em Case 1:} $(1,0)^T\ve w$, $(0,1)^T\ve w$ and $(1,1)^T\ve w$ are all $0$ modulo $\frac1q$. Then $\v \in \intr(J)$ since $J \in \I_{q,2}$. Then one can choose a vector $\ve d$ such that $\ve w' = \ve w + \ve d \in \relint(K)$ and $\v' = \ve v + \ve d \in \intr(J)$. Then $\u \oplus \v' = \ve w'$ and we are done. 

{\em Case 2:} $(1,0)^T\ve v$, $(0,1)^T\ve v$ and $(1,1)^T\ve v$ are all $0$ modulo $\frac1q$.  Then $\ve w \in \intr(K)$ and one can choose again a vector $\ve d$ such that $\ve w' = \ve w + \ve d \in \intr(K)$ and $\v' = \ve v + \ve d \in \intr(J)$. Then $\u \oplus \v' = \ve w'$ and we are done. 

{\em Case 3:} Exactly one of $(1,0)^T\ve w$, $(0,1)^T\ve w$ and $(1,1)^T\ve w$ is $0$ modulo $\frac1q$ and the same holds for $\v$. This means $\ve w$ and $\v$ lie on different hyperplanes in the arrangement $\mathcal{H}_q$. But then one can again choose a vector $\ve d$ such that $\ve w' = \ve w + \ve d \in \relint(K)$ and $\v' = \ve v + \ve d \in \intr(J)$. Then $\u \oplus \v' = \ve w'$ and we are done. 

{\em Case 4:} None of $(1,0)^T\ve w$, $(0,1)^T\ve w$ and $(1,1)^T\ve w$ is $0$ modulo $\frac1q$ and the same holds for $\v$. This means $\v \in \intr(J)$ and $\ve w \in \intr(K)$ already and we are done.

Part (ii) can be proved in a similar way.


\old{and let $\epsilon > 0$ be such that the open ball centered at $\u$ or radius $\epsilon$ is a subset of $I$. Since $(I,J,K)$ is a valid triple, there exist $\v \in J$ and $\ve w \in K$ such that $\u \oplus \v = \ve w$. Choose vectors $\ve d, \ve d'$ of norm at most $\frac{\epsilon}{2}$ such that $\v + \ve d \in \relint(J)$ and $\ve w + \ve d' \in \relint(K)$ (if $\v \in \relint(J)$ one can choose $\ve d = 0$ and if $\ve w \in \relint(K)$ and $\v$ then one can choose $\ve d' = 0$). Now observe that $\u' = \u + \ve d' - \ve d \in \intr(I)$ because both $\ve d, \ve d'$ have norm at most $\frac{\epsilon}{2}$ and $\u' $}
\end{proof}

\old{
\begin{lemma}\label{lemma:interiorDiag}
Suppose $(I,J,K)$ is a valid triple.
\begin{enumerate}[\rm i.]
\item If $I,J \in \I_{q,2}$ and $K \in \Idiag$, then for every point $\v \in \intr(J)$ there exists a point $\u \in \intr(I)$ such that $\u \oplus \v \in K$.
\item If $I,K \in \I_{q,2}$ and $J \in \Idiag$, then for every point $\ve w \in \intr(K)$ there exists a point $\u \in \intr(I)$ such that $\ve w \ominus \u \in J$.
\end{enumerate}
\end{lemma}
\begin{proof}
\emph{Part (i).}
 Let $\v \in \intr(J)$. Since $(I,J,K)$ is a valid triple, there exists a $\u \in I$ such that $\u \oplus \v \in K$.  If $\u \in \intr(I)$, then we are done.  So suppose $\u \notin \intr(I)$.  We first claim that there exists a $\epsilon \in \R$ such that $\u + \epsilon (-1,1)^T \in \intr(J)$.  If not, then $\ve 1 \cdot \x =  \ve 1 \cdot \u$ is a  hyperplane that intersects $I$ only on the boundary.  Therefore, $\ve 1\cdot \x \neq \lambda$ for all $x \in I$ for either all $\lambda > \ve 1 \cdot \u$ or $\lambda < \ve 1 \cdot\u$.   Without loss of generality, assume the sign is $>$.  Since $\v \in \intr(J)$, there exists an $\epsilon > 0$ such that $ \v - \epsilon (1,1)^T \in J$, and there exists a $ \hat \x \in I$ such that $\v - \epsilon (1,1)^T + \hat \x \in K$.  But then we must have $\ve 1 \cdot\hat \x = \ve 1 \cdot \x + 2 \epsilon$, which is not possible.  Therefore, $\u +  \epsilon (-1,1)^T \in \intr(I)$.

Note that $\u - \epsilon (-1,1)^T \notin \I$.  Suppose without loss of generality that $\epsilon > 0$.  By construction of $\P_q$, if $\x \in I$, this implies either that $(1,0)\cdot \x \geq (1,0) \cdot \u$ or $(1,0)\cdot \x \leq (1,0) \cdot \u$.  Without loss of generality, suppose the sign is $\geq$.  
Next, we claim that $\u +  \epsilon (-1,1)^T + \v \in K$.  Suppose not.  Since $\v \in \intr(J)$, $\v + \epsilon (-1,1)^T \in J$, and therefore there exists a $\hat \u \in I$ such that $\v + \epsilon (-1,1)^T + \hat \u \in K$.  But then we must have $ (1,0) \cdot \hat \u < (1,0) \cdot \u$, which is a contradiction.  Therefore, we conclude that $\u +  \epsilon (-1,1)^T \in \intr(I)$ and $\u + \epsilon (-1,1)^T \oplus \v \in K$.  

\emph{Part (ii).}
This part can be proved in a similar way to part (i).
\end{proof}
}

\subsection{Interval lemma}
The so-called Interval
Lemma was introduced by Gomory and Johnson in~\cite{tspace}.  We prove this in a more general setting with three functions by a modifying a proof from \cite{bccz08222222}.
\begin{lemma}[Interval Lemma]\label{one-dim-interval_lemma} Given real numbers $u_1 < u_2$ and $v_1 < v_2$, let $U = [u_1, u_2]$, $V = [v_1, v_2]$, and $U + V = [u_1 + v_1, u_2 + v_2]$.
Let $f : \; U \; \rightarrow
\mathbb{R}$, $g : \; V \; \rightarrow
\mathbb{R}$, $h : \; U+V \; \rightarrow
\mathbb{R}$ be bounded functions. \\ If $f(u)+g(v) = h(u+v)$ for every
$u \in U$ and $v \in V$, then there exists $c\in \R$ such that $f(u)=f(u_1)+c(u-u_1)$ for every $u\in U$, $g(v)=g(v_1)+c(v-v_1)$ for every $v\in V$, $h(w)=h(u_1+v_1)+c(w-u_1-v_1)$ for every $w\in U+V$.
\end{lemma}

\begin{proof} We first show the following.
\medskip

\noindent{\em Claim 1. Let $u \in U$,  and let $\varepsilon>0$ such that $v_1+\varepsilon\in V$. For every nonnegative integer $p$ such that $u+p\varepsilon\in U$, we have $f(u + p\varepsilon) - f(u) = p(g(v_1 + \varepsilon) -
g(v_1))$.}
\medskip

For $h=1, \ldots, p$, by hypothesis $f(u + h\varepsilon) + g(v_1) =
h(u + h\varepsilon + v_1) = f(u+(h-1)\varepsilon) + g( v_1+
\varepsilon)$. Thus $f(u + h\varepsilon) - f(u+(h-1)\varepsilon) = g(v_1 + \varepsilon) -
g(v_1)$, for $h=1,\ldots,p.$
By summing the above $p$ equations, we obtain $f(u + p\varepsilon) - f(u) = p(g(v_1 + \varepsilon) - g(v_1))$. This concludes the proof of Claim~1.\medskip

Let $\bar u, \bar u'\in U$ such that $\bar u-\bar u'\in \Q$ and
$\bar u>\bar u'$.  Define $c:=\frac{f(\bar u)-f(\bar u')}{\bar
u-\bar u'}$.
\bigskip

\noindent{\em Claim 2. For every $u,u'\in U$ such that $u-u'\in \Q$,
we have $f(u)-f(u')=c(u-u')$. }\medskip

We only need to show that, given $u,u'\in U$
such that $u-u'\in \Q$, we have $f(u)-f(u')=c(u-u')$. We may
assume $u>u'$. Choose a positive rational $\varepsilon$ such that
$\bar u-\bar u'=\bar p \varepsilon$ for some integer $\bar p$, $ u-
u'= p \varepsilon$ for some integer $ p$, and $v_1+\varepsilon \in
V$. By Claim~1,
$$f(\bar u)-f(\bar u')=\bar p (g(v_1 + \varepsilon) -
g(v_1)) \quad \mbox{
and
}\quad f( u)-f( u')= p (g(v_1 + \varepsilon) -
g(v_1)).$$
Dividing the last equality by $u-u'$ and the second to last by $\bar u-\bar u'$, we get
$$\frac{g(v_1 + \varepsilon) -
g(v_1)}{\varepsilon}=\frac{f(\bar u)-f(\bar u')}{\bar u-\bar u'}=\frac{f( u)-f( u')}{ u- u'}=c.$$
Thus $f(u)-f(u')=c(u-u')$. This concludes the proof of Claim~2.
\bigskip

\noindent{\em Claim 3. For every $u\in U$, $f(u) = f(u_1) + c
(u-u_1)$.}
\medskip

Let $\delta(x) = f(x) - c  x$ for all $x\in U$. 
We show that  $\delta(u) = \delta(u_1)$ for all $u \in U$ and this proves the claim.  Since
$f$ is bounded on $U$, $\delta$ is bounded over
$U$. Let $M$ be a number such that $|\delta(x)| \leq M$
for all $x\in U$.

Suppose by contradiction that, for some $u^*\in U$, $\delta(u^*)
\neq \delta(u_1)$. Let $N$ be a positive integer such that
$|N(\delta(u^*)-\delta(u_1))| > 2M$.\\ By Claim~2, $\delta(u^*) =
\delta(u)$ for every $u\in U$ such that $u^*-u$ is rational. Thus
there exists  $\bar u$ such  that $\delta(\bar u)=\delta(u^*)$, $u_1
+ N(\bar u - u_1) \in U$ and $v_1 + \bar u - u_1 \in V$. Let $\bar u
- u_1 = \varepsilon$. By Claim~1,
\[
\begin{array}{rcl}
\delta(u_1 + N\varepsilon) - \delta(u_1) &  =  & (f(u_1 + N\varepsilon) - c(u_1+N\varepsilon)) - (f(u_1) - cu_1) \\
& = & N(g(v_1+\varepsilon)-g(v_1)) -c(N\varepsilon) \\
& = & N(f(u_1 + \varepsilon)-f(u_1)) - c(N\varepsilon) \\
& = & N(f(u_1 + \varepsilon)-f(u_1) - c\varepsilon) \\
& = & N(\delta(u_1 + \varepsilon) - \delta(u_1) ) \\
& = & N(\delta(\bar u)-\delta(u_1))
\end{array}
\]
Thus $|\delta(u_1 + N\varepsilon) - \delta(u_1)| =  |N(\delta(\bar
u)-\delta(u_1))|=|N(\delta(u^*)-\delta(u_1))| > 2M$, which implies
$|\delta(u_1 + N\varepsilon)|+ |\delta(u_1)|>2M$, a contradiction.
This concludes the proof of Claim~3.
\bigskip

By symmetry between $U$ and $V$, Claim~3 implies  that there exists
some constant $c'$ such that, for every $v\in V$,
$g(v)=g(v_1)+c'(v-v_1)$. We show $c'=c$. Indeed, given
$\varepsilon>0$ such that $u_1+\varepsilon\in U$ and
$v_1+\varepsilon\in V$, $c\varepsilon=f(u_1 + \varepsilon) -
f(u_1) = g(v_1 + \varepsilon) -g(v_1)=c'\varepsilon$, where
the second equality follows from Claim~1.\\ Therefore, for every
$v\in V$, $g(v) = g(v_1) + c  g(v-v_1)$. Finally, since
$f(u)+g(v)=h(u+v)$ for every $u\in U$ and $v\in V$, we have that for every
$w \in U+V$, $h(w) = h(u_1+v_1) + c  (w-u_1-v_1)$. \end{proof}

\old{\begin{lemma}[Higher Dimensional Interval Lemma] \label{interval_lemma} Let $U$, $V$ be parallelotopes in $\R^k$, i.e., there exist $u_0, u_1, \ldots, u_k \in \R^k$ and $v_0, v_1\ ldots, v_k \in \R^k$ such that $U = \{u_0 + \sum_{i=1}^k\lambda_iu_i : 0 \leq \lambda_i \leq 1\}$ and $V = \{v_0 + \sum_{i=1}^k\mu_iv_i : 0 \leq \mu_i \leq 1\}$. Let $f : \; U \; \rightarrow
\mathbb{R}$, $g : \; V \; \rightarrow
\mathbb{R}$, $h : \; U+V \; \rightarrow
\mathbb{R}$ be bounded functions. \\ If $f(u)+g(v) = h(u+v)$ for every
$u \in U$ and $v \in V$, then there exists $c \in \R^k$ such that $f(u)=f(u_0)+c\cdot(u-u_0)$ for every $u\in \intr(U)$, $g(v)=g(v_0)+c\cdot(v-v_0)$ for every $v\in \intr(V)$, $h(w)=h(u_0+v_0)+c\cdot(w-u_0-v_0)$ for every $w\in \intr(U+V)$.
\end{lemma}}

\subsection{Generalized interval lemma and corollaries}
\label{s:generalized-interval-lemma}


The following lemma is a generalization to higher dimensions of the interval lemma that appears in the literature for the infinite group problem.

\begin{lemma}[Higher Dimensional Interval Lemma]\label{lem:generalized_interval_lemma}
Let $\pi : \R^k \to \R$ be a bounded function. Let $U$ and $V$ be compact convex subsets of $\R^k$ such that $\pi(\u) + \pi(\v) = \pi(\u+\v)$ for all $\u \in U$ and $\v \in V$. Corresponding to every linear subspace $L$ of $\R^k$,  there exists a vector $\g$ in the dual space $L'$ of $L$ with the following property. For any $\u^0 \in U$ and $\v^0 \in V$ such that $\u^0$ (resp. $\v^0$) is in the interior of $(\u^0 + L) \cap U$ (resp. $(\v^0 + L) \cap V$) in the relative topology of $L$, the following conditions hold:
\begin{itemize}
\item[(i)] $\pi(\u^0 + \p) = \pi(\u^0) + \langle \g, \p \rangle$ for all $\p \in L$ such that $\u^0 + \p \in U$.
\item[(ii)] $\pi(\v^0 + \p) = \pi(\v^0) + \langle \g, \p \rangle$ for all $\p \in L$ such that $\v^0 + \p \in V$.
\item[(iii)] $\pi(\u^0 + \v^0 + \p) = \pi(\u^0 + \v^0) + \langle \g, \p \rangle$ for all $\p \in L$ such that $\v^0 + \p \in U +V$.
\end{itemize}

\end{lemma}
\begin{proof}
We fix an arbitrary linear subspace $L$ and exhibit a vector $\g \in L'$ with the stated property. Let $\p^1, \ldots, \p^m$ be a basis for $L$ (we obviously have $m \leq k$). Now consider any $\u^0 \in U$ and $\v^0 \in V$ such that $\u^0$ (resp. $\v^0$) is in the interior of $(\u^0 + L) \cap U$ (resp. $(\v^0 + L) \cap V$) in the relative topology of $L$. Let $u^i_1 < u^i_2 \in \R$, $i = 1, \ldots, m$ be such that the intersection of the line $\u^0 + \lambda \p^i$ with $U$ is given by $\{ u^0 + \lambda p \colon u^i_1 \leq \lambda \leq u^i_2\}$  (these numbers exist since $U$ is assumed to be compact and convex), similarly, $v^i_1 < v^i_2 \in \R$ are defined with respect to $V$, $\v^0$ and $\p^i$. 

Let $f^i : [u^i_1, u^i_2] \to \R$ be defined by $f^i(\lambda) = \pi(\u^0 + \lambda \p^i)$, $g^i : [v^i_1, v^i_2] \to \R$ be defined by $g^i(\lambda) = \pi(\v^0 + \lambda \p^i)$ and $h^i : [u^i_1 + v^i_1, u^i_2 + v^i_2] \to \R$ be defined by $h^i(\lambda) = \pi(\u^0 + \v^0 + \lambda \p^i)$. Applying Lemma~\ref{one-dim-interval_lemma}, there exists a constant $c_i \in \R$ such that 
\begin{equation}\label{eq:affine-prop}
\begin{array}{l}
\pi(\u^0 + \lambda \p^i) = \pi(\u^0) + c_i\cdot \lambda \textrm{ for all } \lambda \in [u^i_1, u^i_2], \\
\pi(\v^0 + \lambda \p^i) = \pi(\u^0) + c_i\cdot \lambda  \textrm{ for all } \lambda \in [v^i_1, v^i_2] \textrm{ and }\\
\pi(\u^0 + \v^0 + \lambda \p^i) = \pi(\u^0 + \v^0) + c_i\cdot \lambda \textrm{ for all }\lambda \in [u^i_1 + v^i_1, u^i_2 + v^i_2].
\end{array}
\end{equation}

Notice that this argument could be made with $\u^0$ and {\em any other} $\v \in V$ with the property that $\v$ is in the interior of $(\v + L) \cap V$. Thus, $c_i$ is independent of $\v^0$. Applying a symmetric argument by fixing $\v^0$ and considering different $\u \in U$, we see that $c_i$ is also independent of $\u^0$. In other words, $c_i$, $i=1, \ldots, m$ only depend on $\pi$, $L$ and the two sets $U$ and $V$, and~\eqref{eq:affine-prop} holds for any $\u \in U$ and $\v \in V$ with the property that $\u$ (resp. $\v$) is in the interior of $(\u + L) \cap V$ (resp. $(\u + L) \cap V$) in the relative topology of $L$.

We choose $\g \in L'$ as the unique dual vector satisfying $\langle \g, \p^i \rangle = c_i$. Now for any $\p \in L$ such that $\u^0 + \p \in U$. We can then represent $\p = \sum_{i=1}^m \lambda_i \p^i$ such that $u^i_1 \leq \lambda_i \leq u^i_2$. Thus, $\begin{array}{rcl} \pi(\u^0 + \p) & = & \pi(\u^0 + \sum_{i=1}^m \lambda_i \p^i)\end{array}.$
Now using~\eqref{eq:affine-prop} with $i = m$ we have 
$$\begin{array}{rcl} \pi(\u^0 + \sum_{i=1}^m \lambda_i \p^i) & = &\pi(\u^0 + \sum_{i=1}^{m-1}\lambda_i \p^i + \lambda_m \p^m) \\
& = & \pi(\u^0 + \sum_{i=1}^{m-1}\lambda_i \p^i) + c_m\cdot\lambda_m 
\end{array}
$$
which follows because the $c_i$'s do not depend on the particular point $\u^0$ and in the case above we apply it on the point $\u^0 + \sum_{i=1}^{m-1}\lambda_i \p^i$. By applying this argument iteratively, we find that

$$\begin{array}{rcl} \pi(\u^0 + \sum_{i=1}^m \lambda_i \p^i) & = & \pi(\u^0) + \sum_{i=1}^{m}c_i\cdot \lambda_i \\
& = & \pi(\u^0) + \sum_{i=1}^{m} \langle \g, \p^i \rangle \cdot \lambda_i \\
& = &  \pi(\u^0) + \langle \g,  \sum_{i=1}^{m} \lambda_i \p^i \rangle \\
& = &  \pi(\u^0) + \langle \g, \p \rangle
\end{array}
$$

This proves condition $(i)$ in the statement of the lemma. The same argument applies for proving conditions $(ii)$ and $(iii)$. \end{proof}

\old{The first corollary we state strengthens the hypothesis on $\pi$ to be continuous and bounded, but relaxes the hypothesis on the additivity relations.

\begin{corollary} Let $\pi : \R^k \to \R$ be a continuous and bounded function. Let $U$ and $V$ be compact convex subsets of $\R^k$ and $W$ be a closed subset of $U+V$, such that $\pi(\u) + \pi(\v) = \pi(\u+\v)$ for all $\u \in U$ and $\v \in V$ with $\u + \v \in W$. Let $L$ be a subspace of $\R^k$ and let $R(L) \subseteq \R^k \times \R^k \times \R^k$ be the set of all tuples $(\u,\v,\w)$ such that $\u$ is in the interior o

 satisfying the following: for every $\u \in U$ such that $\u$ is in the interior of $\u + L \cap U$ in the relative topology of $L$, there exists

Corresponding to every linear subspace $L$ of $\R^k$,  there exists a vector $\g$ in the dual space $L'$ of $L$ with the following property. For any $\u_0 \in U$, $\v_0 \in V$ and $\w_0 \in W$ such that $\u_0$ (resp. $\v_0$, $\w_0$) is in the interior of $(\u_0 + L) \cap U$ (resp. $(\v_0 + L) \cap V$, $(\w_0 + L) \cap W$) in the relative topology of $L$, the following conditions hold:
\begin{itemize}
\item[(i)] $\pi(\u_0 + \p) = \pi(\u_0) + \langle \g, \p \rangle$ for all $\p \in L$ such that $\u_0 + \p \in U$.
\item[(ii)] $\pi(\v_0 + \p) = \pi(\v_0) + \langle \g, \p \rangle$ for all $\p \in L$ such that $\v_0 + \p \in V$.
\item[(iii)] $\pi(\w_0 + \p) = \pi(\u_0 + \langle \g, \p \rangle$ for all $\p \in L$ such that $\w_0 + \p \in W$.
\end{itemize}
\end{corollary}

\begin{proof}
Given $\ve u_0 \in U$ and $\ve v_0 \in V$ satisfying the relative topology condition, there exists a closed ball $B_1, B_2 \subset L$ such that $\ve u_0$ (resp. $\ve v_0$) is in the interior of $B_1$ (resp. $B_2$) in the relative topology of $L$, such that $$
\end{proof}
}
\old{
\texttt{Don't these corollaries need the patching arguments? ( I pasted previous lemma here) -Robert}
{ \tt 
\begin{lemma}
$\pi$ is affine imposing in $\I^2_{G,1}$.
\end{lemma}
}
\begin{proof}
{\tt 
Consider any $I \in \I^2_{G,1}$. By definition, there exists $F\in
\P^{0,2}_\B$ such that $I \subseteq (1,0)F$ or $I \subseteq (1,1)F$. We will
show the result for $I \subseteq (1,0) F$; the proof for $I \subseteq (1,1) F$
is similar. 
Fix any $x_0 \in \intr(I)$. We will show that there exists $c \in \R$ such that $\pi^1(y) = \pi^1(x_0) + c\cdot (y-x_0)$ for all $y \in \intr(I)$. This will prove the claim.

Let $y \in \intr(I)$.  Since $x_0,y \in \intr(I)$, there exist points $(p^0_1,
p^0_2), (q_1, q_2) \in \relint(F)$ such that $x_0 = (1,0)\cdot (p^0_1, p^0_2)$
and $y = (1,0)\cdot (q_1, q_2)$. 
 Therefore, we can construct a sequence of closed intervals $U_0, \ldots, U_n
 \subseteq \intr(I)$ and another sequence of closed intervals $V_0, \ldots, V_n$ such that $U_i \times V_i \subseteq \relint(F)$ and $(p^0_1, p^0_2) \in U_0\times V_0$ and $(q_1, q_2) \in U_n \times V_n$ such that $\intr(U_{i-1} \times V_{i-1}) \cap \intr(U_{i} \times V_{i}) \neq \emptyset$, for all $i = 1, \ldots, n$. 
Therefore, we can find a sequence of points $(p^1_1, p^1_2), \ldots, (p^n_1, p^n_2)$ such that $(p^i_1, p^i_2) \in\intr(U_{i-1} \times V_{i-1}) \cap \intr(U_{i} \times V_{i})$. Let $x_i = (1,0)\cdot (p^i_1, p^i_2)$ for all $i=1,\ldots, n$.

Now, since $\Delta\pi(x,y) = 0$ over $\relint(F)$, we have that $\Delta\pi(u,v) = 0$ for all $(u,v) \in U_i\times V_i$, $i = 0, \ldots, n$. This implies $\pi(u) + \pi(v) = \pi(u\oplus v)$ for all $(u,v) \in U_i\times V_i$, and so the same relation holds for $\pi^1$. Using the Interval Lemma, there exists $c_i$ such that $\pi^1(u) = \pi^1(x_i) + c_i\cdot(u - x_i)$, for all $u \in U_i$ and this holds for every $i = 0, \ldots, n$. Observe that $x_i$ belongs to the interior of both $U_{i-1}$ and $U_i$ for all $i= 1, \ldots, n$. Therefore, we must have $c_{i-1} = c_i$ for all $i = 1, \ldots, n$; we take $c$ to be this common value. Now using the relation $\pi^1(y) = \pi^1(x_0) + (\sum_{i=1}^n \pi^1(x_i) - \pi^1(x_{i-1})) + (\pi^1(y) - \pi^1(x_n))$, we find that $\pi^1(y) = \pi^1(x_0) + c\cdot(y - x_0)$.
}
\end{proof}

Define the function $\bar \theta : \R^2 \times \R^2 \to \R^2$ as $\bar\theta (\u,\v) = \u + \v$ for all $\u,\v \in \R^2$. Observe that $\bar\theta$ is a continuous function. Let $\proj_U : \R^2 \times \R^2 \to \R^2$ be defined as $\proj_U(\u,\v) = \u$ and $\proj_V : \R^2 \times \R^2 \to \R^2$ be defined as $\proj_U(\u,\v) = \v$.
}

Now the lemmas stated in \autoref{s:real-analysis} follow as corollaries.

\begin{proof}[of \autoref{cor:triangle+triangle}]
Let $U(\ve x, r)\subseteq \R^2$ denote the $\ell_\infty$ ball of radius $r$ around $\ve x \in \R^2$. Define 
$$r(\ve u) = \sup\{r \in \R \colon \exists \ve v, \ve w \textrm{ such that } U(\ve u, r) \subseteq I, U(\ve v, r) \subseteq J, U(\ve w, 2r)\subseteq K\}.$$ Since $(I,J,K)$ is a valid triple, by Lemma~\ref{lemma:interiorDiag} (i), for any $\ve u \in \intr(I)$, there exist $\ve v \in \intr(J)$ and $\ve w \in \intr(K)$ such that $\ve u \oplus \ve v = \ve w$. Thus, $r(\u) > 0$ for every $\u \in I$.

\begin{claim} 
$r(\ve u)$ is a continuous function of $\ve u$.
\end{claim}

\begin{proof}
$r(\ve u)$ is the optimal value of the linear program with variables $r, \ve v, \ve w$ given by $$\max r \textrm{ subject to } \ve u \oplus \ve v = \ve w, U(\ve u, r) \subseteq I, U(\ve v, r) \subseteq J, U(\ve w, 2r) \subseteq K.$$ All the constraints can be written as linear constraints. Since the value of a parametric linear program is continuous in the parameter (in this case the parameter is $\ve u$) we are done.
\end{proof}

We will now show that for any two points $\ve x_1, \ve x_2 \in \intr(I)$, there exist finitely many full-dimensional parallelotopes $U_1, \ldots, U_k$ in $\R^2$ such that $\ve x_1 \in U_1$, $\ve x_2 \in U_k$ and $\intr(U_i) \cap \intr(U_{i+1}) \neq \emptyset$ for all $i = 1, \ldots, k-1$. Moreover, we will show that $\pi$ is affine over each $U_i$, $i=1, \ldots, k$. This will imply that in fact $\pi$ is affine over $\intr(I)$ and therefore, by continuity, over $I$. By a symmetric argument, one can show that $\pi$ is affine over $J$. This will then show that $\pi$ is affine over $K$.

Given $\ve x_1, \ve x_2 \in \intr(I)$, consider the minimum value $\epsilon$
of $r(\ve u)$ as $\ve u$ varies over the line segment $[\ve x_1, \ve
x_2]$. Note that $\epsilon$ is strictly greater than 0 as it is the minimum of
a strictly positive function over a compact set. This implies that we can find
a set of points $\u_1 =\ve x_1, \ldots, \u_k = \ve x_2$ on the line segment
$[\ve x_1, \ve x_2]$ such that if we let $U_i = U(\u_i, \epsilon)$ we have the
property that $\ve x_1 \in U_1$, $\ve x_2 \in U_k$ and $\intr(U_i) \cap
\intr(U_{i+1}) \neq \emptyset$ for all $i = 1, \ldots, k-1$. Now, by the
definition of $r(\ve u_i)$ which is greater than or equal to $\epsilon$, there
exist $\ve v_i$ and $\ve w_i$, $i = 1, \ldots, k$ such that $U(\ve u_i,
r(u_i)) \in I, U(\ve v_i, r(\u_i)) \in J, U(\ve w_i, 2r(\ve u_i)) \in
K$. Applying Lemma~\ref{lem:generalized_interval_lemma}, with $L = \R^2$, $U =
U(\ve u_i, r(\ve u_i))$, $V =  U(\ve v_i, r(\ve u_i))$ and $\u_0 = \u_i$ and
$\v_0 \in v_i$, we obtain that $\pi$ is affine over $U(\u_i, r(\ve u_i))$ and hence over $U_i \subseteq U(\u_i, r(\u_i))$.

The fact that the gradient over $I$ and $J$ (and hence over $K$) are the same follows from the observation that Lemma~\ref{lem:generalized_interval_lemma} gives the same gradient over the parallelotopes $U = U(\ve u_i, r(\u_i))$ and $V =  U(\ve v_i, r(\u_i))$ in the above argument.
\end{proof}

Similar arguments can be used to show \autoref{cor:triangle+diagonal}.  We
omit the proof.  
\old{\begin{proof}
In Lemma~\ref{lem:generalized_interval_lemma}, we can pick $L$ to be linear space spanned by $(-1,1)^T$, $U = I$, $V = J$ and any $\u_0 \in \relint(I)$ and $\v_0 \in \relint(J)$.
\end{proof}}

\subsection{Transferring affine linearity}
\label{s:transferring-affine-linearity}

\begin{proof}[of \autoref{obs:adjacent}]
Let $e\in \P_{q,1}$ be the common edge for $I$ and $J$. We assume that $e$ is horizontal (the argument for vertical edges is exactly the same) and let $\v^0 \in \R^2$ be the vertex of $e$ such that the other vertex is $\v^0 + (0,1)^T$.  Since $\pi$ is affine on $I$, there exists $c' \in \R$ such that $\pi(\v^0 + \lambda (0,1)^T) = \pi(\v^0) + c' \cdot\lambda$ for all $0 \leq \lambda \leq 1$. Now observe that any point in $J$ can be written as $\v^0 + \mu_1 (0,1)^T + \mu_2 (-1,1)^T$ with $0 \leq \mu_1, \mu_2 \leq 1$ and therefore, $\pi(\v^0 + \mu_1 (0,1)^T + \mu_2 (-1,1)^T) = \pi(\v^0 + \mu_1 (0,1)^T) + c\cdot\mu_2$ (using $(ii)$ in the hypothesis) and $\pi(\v^0 + \mu_1 (0,1)^T) + c\cdot\mu_2 = \pi(\v^0) + c'\cdot\mu_1 + c\cdot\mu_2$. Thus, $\pi$ is affine on $J$.
\end{proof}

\begin{proof}[of \autoref{lemma:point-and-line-Lemma}]
\emph{Case (i).}  Suppose $\{I,J\}\in \E_0$.
Since $\pi, \pi^1, \pi^2$ are all continuous, we just prove that $\partial_\v$ is constant on $\intr(J)$.
If $\{I,J\} \in \E$, $\exists \a \in \frac{1}{q}\Z^2$ such that, setting $K = \{\a\} \in \I_{q,0}$, one of the following two cases occurs.\par\noindent
\textbf{Case 1.} $(I,J,K) \in E(\pi, \P_q)$.   Then $\pi_I(\x) + \pi_J(\y) = \pi_K(\a)$ for all $\x \in I, \y \in J, \x \oplus \y = \a$, or rewriting this, we have
$\pi_J(\x) = \pi_K(\a) - \pi_I(\a \ominus \x)$.  For any $\u \in \intr(J)$, it follows from Lemma \ref{lemma:translate} that $a \ominus \u \in \intr(I)$.
Since the right hand side is differentiable in the direction of $\v$ at $\a \ominus \u$, the left hand side is as well.  The result in this case follows by the chain rule.  
%
\par\noindent 
 \textbf{Case 2.} $(I,K,J) \in E(\pi, \P_q)$.   Then $\pi_I(\x) + \pi_K(\a) = \pi_J(\y)$ for all $\x \in I, \y \in J, \x \oplus \a = \y$, or rewriting this, we have
$\pi_J(\x \oplus \a) = \pi_I(\x) + \pi_K(\a)$.  For any $\u \in \intr(J)$, it follows from Lemma \ref{lemma:translate} that $\u \ominus \a \in \intr(I)$. Since the right hand side is differentiable in the direction of $\v$, the left hand side is as well.  The result again follows by the chain rule.  
%

\emph{Case (ii).} Suppose $\{I,J\}\in \E_\smallsetminus$.  Using Lemma
\ref{lemma:interiorDiag}, the proof follows similar to
Case~(i).
\end{proof}

\subsection{Proof of \autoref{theorem:systemNotUnique}}
\label{s:proof-theorem:systemNotUnique}

\begin{proof}[of \autoref{theorem:systemNotUnique}]
  \emph{Part (i).}
Suppose \eqref{equation:system} does not have a unique solution. Let $\bar \varphi\colon \tfrac{1}{q} \Z^2\to\R$ be a non-trivial element in the kernel of the system above.  Then for any $\epsilon$, $\pi|_{\tfrac{1}{q} \Z^2} + \epsilon \bar\varphi$ also satisfies the system of equations. 
Let
$$
\epsilon = \min \{\,\Delta\pi_F(\x,\y)\neq 0 \st F \in \Delta \P_q, \ (\x,\y) \in \verts(F)\,\}.
$$ 
Let $\bar\pi\colon \R^2 \to \R$ be the continuous piecewise linear extension of $\varphi$ over $\P_q$ and set $\pi^1 = \pi + \tfrac{\epsilon}{3 ||\bar\pi||_\infty}
\bar\pi$, $\pi^2 = \pi - \tfrac{\epsilon}{3 ||\bar\pi||_\infty}
\bar\pi$.  Note that $0 < ||\bar\pi||_\infty < \infty$ since
$\bar\varphi$ comes from a non-trivial element in the kernel.  We claim that $\pi^1, \pi^2$ are
both minimal.  As before, we show this for $\pi^1$, and the proof for $\pi^2$ is similar.  
Since $\pi|_{\tfrac{1}{q}\Z^2}$ satisfies the system~\eqref{equation:system} and $\bar\varphi$ is an
element of the kernel, $\pi^1$ satisfies the system~\eqref{equation:system} as
well. In particular, we have $\pi^1(\0) = 0, \pi^1(\f)= 1, \pi^1((0,1)) = 0,  \pi^1((0,1)) = 0, \pi^1((1,1)) = 0$. 

Next,  $\pi^1$ is symmetric because the symmetry conditions are implied here, that is,  since we require that $\varphi(\f) = 1$, and since $\pi$ is minimal, $\Delta\pi(\u, \f- \u) = 0$ whenever $\u \in \frac{1}{q}\Z^2$, hence, by Theorem \ref{minimality-check}, $\pi^1$ is symmetric.  

Lastly, we show that $\pi^1$ is subadditive.  Let $\u,\v \in \frac{1}{q}\Z^2$.  If $\Delta\pi(\u,\v) = 0$, then $\Delta\varphi(\u,\v) = 0$, as implied by the system of equations.  Otherwise, if $\Delta\pi(\u,\v) > 0$, then   
\begin{align*}
  \Delta\pi^1(\u,\v) 
  &= \Delta\pi(\u,\v) + \frac{\epsilon}{3
    ||\bar\varphi||_\infty} \bar\varphi(\u) + \frac{\epsilon}{3
    ||\bar\varphi||_\infty} \bar\varphi(\v) - \frac{\epsilon}{3
    ||\bar\varphi||_\infty} \bar\varphi(\u\oplus \v) \\
  &\geq \Delta\pi_F(\u,\v) -
  \tfrac{\epsilon}{3 } - \tfrac{\epsilon}{3} - \tfrac{\epsilon}{3 } \geq 0
\end{align*}
Therefore, by Theorem~\ref{minimality-check}, $\pi^1$ (and $\pi^2$) is
subadditive and therefore minimal and valid.  Therefore $\pi$ is not
extreme.\smallbreak

\emph{Part (ii).}
Suppose there exist distinct, valid functions $\pi^1, \pi^2$ such that $\pi =
\tfrac{1}{2} \pi^1 +\tfrac{1}{2} \pi^2$.  Since $\pi$ is minimal and affine
imposing in $\I_{q,2}$, $\pi^1,\pi^2$ are minimal continuous piecewise linear functions over $\P_q$.  Furthermore, $\pi|_{\tfrac{1}{q} \Z^2}$ and, also $\pi^1|_{\tfrac{1}{q} \Z^2}, \pi^2|_{\tfrac{1}{q} \Z^2}$ satisfy the system
of equations \eqref{equation:system}.  If this system has a unique solution,
then $\pi = \pi^1  = \pi^2$, which is a contradiction since $\pi^1, \pi^2$
were assumed distinct.  Therefore $\pi$ is extreme. 

On the other hand, if the system \eqref{equation:system} does not have a unique solution, then by Theorem \ref{theorem:systemNotUnique}, $\pi$\ is not extreme.
\end{proof}

\end{document}